\documentclass{amsart}
\usepackage{amssymb,amsmath,graphicx,mathrsfs,mathabx,enumerate,tikz,url}
\usepackage{srcltx}
\usetikzlibrary{matrix}

\newtoks\prt 

\newtheorem{thm}{Theorem}[section]
\newtheorem{ques}[thm]{Question} 
\newtheorem{lemma}[thm]{Lemma} 
\newtheorem{prop}[thm]{Proposition} 
\newtheorem{cor}[thm]{Corollary}

\theoremstyle{definition} 
 
\newtheorem*{remark}{Remark}

\def\eqn#1$$#2$${\begin{equation}\label#1#2\end{equation}} 

\let\tilde\widetilde

\def\B{\mathcal B} 
\def\C{\mathcal C} 
 
\def\Dc{\mathscr D} 
\def\Dcp{\mathscr D^{\,\prime}}

\def\F{\mathcal F}

\def\M{\mathcal M} 
\def\Ms{\mathcal M_\sigma}
\def\Mp{\mathcal M_{pfa}}
\def\Ma{\mathcal M_{ac}}
\def\Mn{\mathcal M_s}

\def\bmu{{\boldsymbol  \mu}}
\def\bnu{{\boldsymbol  \nu}}
\def\o{{\boldsymbol o}}

\def\g{\boldsymbol g}
\def\f{\boldsymbol f}
\def\x{\boldsymbol x}
\def\y{\boldsymbol y} 
\def\h{\boldsymbol h} 
\def\bR{\boldsymbol R} 
\def\z{\boldsymbol z}
\def\u{\boldsymbol u}
\def\vv{\boldsymbol v} 
\def\bA{\boldsymbol A} 
\def\bS{\boldsymbol S} 
 
\def\en{\mathbb N} 
\def\er{\mathbb R} 
\def\qe{\mathbb Q}

\def\dist{\operatorname{dist}}

\def \div {\operatorname{div}}

\def \Lip {\operatorname{Lip}}

\def \Int {\operatorname{Int}}

\def\spt{\operatorname{spt}} 
\def\lin{\operatorname{span}}

\def \reg {\partial _{\kern1pt\text{reg}}}

\newcommand{\ip}[2]{\left\langle#1,#2\right\rangle}

\def\di{\,\mbox{\rm d}}

\newcommand{\norm}[1]{\left\|#1\right\|}

\newcommand{\esssup}{\operatorname{ess sup}}

\newcommand{\wscl}[1]{\overline{#1}^{w^*}}

\newcommand{\bchi}{\mbox{\Large$\chi$}}

\newcommand{\abs}[1]{\left|#1\right|}
\newcommand{\setsep}{;\,}

\begin{document}
\title[Finitely additive measures and Lipschitz-free spaces]{Finitely additive measures and complementability of Lipschitz-free spaces}
\author{Marek C\'uth, Ond\v{r}ej F.K. Kalenda and Petr Kaplick\'y}

\address{Department of Mathematical Analysis \\
Faculty of Mathematics and Physic\\ Charles University\\
Sokolovsk\'{a} 83, \\ 186 75, Praha 8, Czech Republic}
\email{cuth@karlin.mff.cuni.cz}
\email{kalenda@karlin.mff.cuni.cz}
\email{kaplicky@karlin.mff.cuni.cz}


\thanks{M.~C\'uth is a junior researcher in the University Center for Mathematical Modelling, Applied Analysis and Computational Mathematics (MathMAC). M.~C\'uth and O.~Kalenda were supported in part by the grant GA\v{C}R 17-00941S. P. Kaplick\'y is a member of the Ne\v cas Center for Mathematical Modeling.}

\begin{abstract} 
We prove in particular that the Lipschitz-free space over a finitely-dimensional normed space
is complemented in its bidual. For Euclidean spaces the norm of the respective projection is $1$. As a tool to obtain the main result we establish several facts on the structure of finitely additive measures on finitely-dimensional spaces.
\end{abstract}

\maketitle

\section{Introduction}

The Lipschitz-free space over a metric space $M$ is a Banach space $\F(M)$ whose linear structure in a way reflects the metric structure of $M$. In the framework of Banach spaces, Lipschitz-free spaces can be used to get a linearization of certain properties. For example, one of the first results of this kind, due to G. Godefroy and N. Kalton \cite{GoKa03} says that, whenever a separable Banach space $Y$ admits (as a metric space) an isometric embedding into a Banach space $X$ there exists a linear isometric embedding. Lipschitz-free spaces form nowadays an active field of research, there are many results on Banach space properties of them -- papers \cite{GoOz14, PeSm15, HaPe14, LaPe13, HaLaPe16, Da15a, Da15b, CuDo16} contain results on approximation properties and Schauder bases, in \cite{CuDoWo16} it is proved that any Lipschitz-free space contains a copy of $\ell^1$ as soon as it has infinite dimension. In \cite{CKK} we provided an isometric description of Lipschitz-free space over convex domains in finite-dimensional normed spaces. In the present paper we will use this representation to investigate complementability of Lipschitz-free spaces in the bidual. 
One of the motivations for this research is \cite[Problem 16]{GoLaZi14} asking whether $\F(\ell^1)$ is complemented in the bidual. This problem is of a particular interest, because a positive answer would solve famous open problem of whether every Banach space which is Lipschitz-isomorphic to $\ell_1$ is actually linearly isomorphic to $\ell_1$, see \cite[comment after Problem 16]{GoLaZi14}. We were not able to answer this question, so we started by the investigation of finite-dimensional spaces. Our main result reads as follows.

\begin{thm}\label{T:main}
Let $E$ be a normed space of a finite dimension $d\ge2$. Then there is a linear projection $Q:\F(E)^{**}\to\F(E)$ such that $\norm{Q}\le d_{BM}(E,\ell^2_d)$, where $d_{BM}$ denotes the Banach-Mazur distance.
\end{thm}

We excluded the case $d=1$ as it is well known and easy. Indeed, by \cite[p. 128]{GoKa03}
$\F(\er)$ is isometric to $L^1(\er)$.
Moreover, $L^1(\er)$ is even $1$-complemented in its bidual, in fact it is an $L$-embedded space, i.e., there is a projection $P:(L^1(\er))^{**}\to L^1(\er)$ satisfying $\norm{x^{**}}=\norm{P(x^{**})}+\norm{x^{**}-P(x^{**})}$ for $x^{**}\in(L^1(\er))^{**}$. For $d\ge2$ the situation is complicated, it is not clear whether $\F(E)$ is $L$-embedded when  $\dim E\ge 2$ (at least in some norm).

An easy consequence is the following statement.

\begin{cor}
Let $E$ be a finite-dimensional normed space and let $M\subset E$ be a subset whose closure has nonempty interior in $E$. Then $\F(M)$ is complemented in $\F(M)^{**}$.
\end{cor}

\begin{proof} By \cite[Corollary 3.5]{ka15}, $\F(M)$ is isomorphic to $\F(E)$, so the statement follows immediately from Theorem~\ref{T:main}.
\end{proof}

Another consequence of the main result is the following (the second assertion follows from the John 
theorem \cite{john}).  

\begin{cor}\
\begin{itemize}
\item Let $E$ be a finite-dimensional Euclidean space. Then $\F(E)$ is $1$-comple\-men\-ted in its bidual.
\item Let $E$ be a normed space of finite dimension $d\ge2$. Then $\F(E)$ is $\sqrt{d}$-complemented in its bidual.
\end{itemize}
\end{cor}

Note that for non-Euclidean spaces the estimate of the norm of a projection depends on dimension. The following question seems to be open.

\begin{ques}
Is there a constant $C\ge1$ such that $\F(E)$ is $C$-complemented in $\F(E)^{**}$ for any finite-dimensional normed space $E$?
\end{ques}

Note that the isometric representation of $\F(\Omega)$ given in \cite{CKK} works for any nonempty open convex subset $\Omega$ of a finite-dimensional normed space $E$. However, the estimate for the projection given in Theorem~\ref{T:main} is proved only in the case $\Omega=E$. Hence, the following question seems to be open.

\begin{ques}\label{Q:Omega}
Does the statement of Theorem~\ref{T:main} hold also for $\F(\Omega)$, where $\Omega$ is a nonempty open convex subset of a finite-dimensional normed space?
\end{ques}

The original motivation for our investigation was, as mentioned above, the problem of complementability of $\F(\ell^1)$ in the bidual. We have not answered this question, but in view of the main results the following problem seems to be natural.

\begin{ques}
Is $\F(\ell^2)$ complemented in $\F(\ell^2)^{**}$? Is it $1$-complemented?
\end{ques}

Let us now describe the basic strategy of the proof. Let $E$ be a finite-dimensional normed space and $\Omega\subset E$ be a nonempty open convex set. In \cite{CKK} it is shown that $\F(\Omega)$ is isometric to the quotient of $L^1(\Omega,E)$ by the subspace of the vector fields with zero divergence (see Proposition~\ref{P:popisCKK} for a precise formulation). So, to prove that $\F(\Omega)$ is complemented in the bidual it is enough to show that the mentioned quotient is complemented in its bidual. This bidual can be described, using standard Banach-space duality, as the respective quotient of $L^1(\Omega,E)^{**}$. The space $L^1(\Omega,E)^{**}$ can be represented as the space of $E$-valued finitely additive measures on $\Omega$ which are zero on Lebesgue null sets (see Section~\ref{s:vector}). Therefore, we start by investigating the finitely additive measures and some natural projections on them.
Section~\ref{s:scalar} is devoted to scalar (real-valued) measures, Section~\ref{s:vector} to vector-valued ones. In Section~\ref{s:bidual} we describe the isometric representation of $\F(\Omega)$ and $\F(\Omega)^{**}$ and prove several results on approximation in case $\Omega=E$. In Section~\ref{s:divergence} we study (not necessarily absolutely continuous) $\sigma$-additive vector-valued measures whose distributional divergence coincides with the divergence of an $L^1$-vector field. We give a characterization of such measures and prove several estimates. In Section~\ref{s:tecny} we construct a projection onto this subspace of measures, using tangent spaces of measures. Finally, by composing three natural mappings we get the sought projection.

\section{Preliminaries}

In this section we collect basic definitions and notation used throughout the paper.

We start by defining the Lipschitz-free spaces. Let $(M,d)$ be a metric space with a distinguished point $0\in M$. By $\Lip_0(M)$ we denote the space of all real-valued Lipschitz functions on $M$ which map the distinguished point $0\in M$ to $0\in\er$. If we equip $\Lip_0(M)$ by the Lipschitz-constant norm $\norm{\cdot}_{lip}$, it becomes a Banach space. For any $x\in M$ let $\delta(x)\in \Lip_0(M)^*$ be the evaluation functional $\delta(x):f\mapsto f(x)$. Then $\delta:M\to \Lip_0(M)^*$ is an isometry and the Lipschitz-free space $\F(M)$ is defined as the closed linear span of $\delta(M)$ in $\Lip_0(M)^*$.

The space $\F(M)$ has the following universal property: Given any Banach space $Z$ and any Lipschitz mapping $f:M\to Z$ satisfying $f(0)=0$, there is a unique bounded linear operator $L:\F(M)\to Z$ such that $f=L\circ\delta$. Moreover, the norm of $L$ equals the Lipschitz constant of $f$. Applying this to $Z=\er$ we see that the dual of $\F(M)$ is canonically isometric to $\Lip_0(M)$.
(Note that by a recent result of N.~Weaver \cite{We17}, the space $\Lip_0(M)$ has a unique predual -- namely $\F(M)$ -- whenever $M$ is either bounded or it is a convex subset of a normed space.) We stress that the weak$^*$ topology on $\Lip_0(M)$ coincides on bounded sets with the topology of pointwise convergence.

We will further use the standard calculus of distributions. Let us recall basic notation and some useful facts. If $\Omega\subset\er^d$ is a nonempty open set, $\Dc(\Omega)$ denotes the space of \emph{test functions} on $\Omega$, i.e., the space of real-valued $\C^\infty$ functions with compact support in $\Omega$.
By $\Dcp(\Omega)$ we denote the space of distributions on $\Omega$.

An \emph{approximate unit} in $\Dc(\er^d)$ is a sequence $(u_n)$ in $\Dc(\er^d)$ given by the formula
$$u_n(\x)=n^d\phi(n\x),\quad\x\in\er^d,$$
where $\phi$ is a fixed nonnegative test function satisfying $\int_{\er^d}\phi=1$.

The approximate unit is used to smoothen measurable functions using convolution. We will apply it also to vector-valued functions. If $X=(\er^k,\norm{\cdot}_X)$ is a finite-dimensional normed space and $\f=(f_1,\dots,f_k):\er^d\to X$ is a locally integrable vector field, we set $u_n*\f=(u_n*f_1,\dots,u_n*f_k)$. 
The following lemma summarizes several results on approximation which follow immeadiately from their well-known scalar versions.

\begin{lemma}\label{L:aproximateunitconvergence} Let $(u_n)$ be an approximate unit in $\Dc(\er^d)$ and let $X=(\er^k,\norm{\cdot}_X)$ be a finite-dimensional normed space.
\begin{itemize}
\item[(i)] If $\f\in L^1(\er^k,X)$, then $u_n*\f\to\f$ in the norm of  $L^1(\er^k,X)$.
\item[(ii)] If $\f:\er^k\to X$ is continuous, then $u_n*\f\to\f$ uniformly on compact sets, in particular pointwise.
\item[(iii)]  If $\f:\er^k\to X$ is uniformly continuous (in particular, if $\f\in\C_0(\er^k,X)$), then $u_n*\f\to\f$ uniformly on $\er^k$.
\end{itemize}
\end{lemma}

We will use also the following easy estimate.

\begin{lemma}\label{L:aproxunitodhad} Let $(u_n)$ be an approximate unit in $\Dc(\er^d)$, let $X=(\er^k,\norm{\cdot}_X)$ be a finite-dimensional normed space and let $\f:\er^d\to X$ be a locally bounded measurable mapping.
\begin{itemize}
\item[(i)] $\norm{(u_n*\f)(\x)}_X\le \sup_{\y\in\x-\spt u_n} \norm{\f(\y)}_X$ for any $\x\in\er^d$.
\item[(ii)] If $\f$ is bounded, then $\norm{u_n*\f}_{L^\infty(\er^d,X)}\le\norm{\f}_{L^\infty(\er^d,X)}$.
\end{itemize}
\end{lemma}

\begin{proof}
The assertion (ii) follows immeadiately from (i). The assertion (i) is easy, we give a proof for the sake of completeness. Fix $\x\in\er^d$. Let $\z^*\in X^*$ be such that $\norm{\z^*}_{X^*}\le 1$. Then we have
$$\begin{aligned}
\abs{\ip{z^*}{(u_n*\f)(\x)}}&=\abs{\ip{z^*}{\int_{\er^d}u_n(\y)\f(\x-\y)\di\y}}\\&=
\abs{\int_{\er^d}u_n(\y)\ip{z^*}{\f(\x-\y)}\di\y}
\\&\le\int_{\er^d} u_n(\y)\norm{\f(\x-\y)}_X\di\y\le \sup_{\y\in\x-\spt u_n} \norm{\f(\y)}_X
\end{aligned}$$
\end{proof}

\section{Finitely additive measures -- scalar case}\label{s:scalar}

Throughout this section $\Omega$ will be a fixed nonempty open subset of $\er^d$ where $d\ge 1$ is fixed. By $\B(\Omega)$ we denote the Borel $\sigma$-algebra on $\Omega$, by $\lambda$ the $d$-dimensional Lebesgue measure on $\Omega$. Further, $\M(\Omega)$ will denote the Banach space of all real-valued finitely additive Borel measures on $\Omega$ equipped with the total variation norm, i.e., 
\begin{equation*}
\norm{\mu}=\sup \{\sum_{i=1}^n\abs{\mu(B_i)}\setsep B_1,\dots,B_n\in\B(\Omega)\mbox{ disjoint} \},\qquad \mu\in\M(\Omega).
\end{equation*}
The space $\M(\Omega)$ is canonically isometric to the dual space to $\B_b(\Omega,\er)$, the space of all bounded real-valued Borel measurable functions on $\Omega$ equipped with the sup-norm, see \cite[Theorem IV.5.1]{DS}. Therefore, there is a canonical weak$^*$ topology on $\M(\Omega)$. For any $\mu\in\M(\Omega)$ we define measures $\mu^+$, $\mu^-$ and $\abs{\mu}$ by
$$\begin{aligned}
\mu^+(B)&=\sup\{\mu(C)\setsep C\subset B\mbox{ Borel}\},\\
\mu^-(B)&=\sup\{-\mu(C)\setsep C\subset B\mbox{ Borel}\},\\
\abs{\mu}(B)&=\sup \{\sum_{i=1}^n\abs{\mu(C_i)}\setsep C_1,\dots,C_n\subset B\mbox{ disjoint Borel sets} \}
\end{aligned}
$$
Then $\mu^+,\mu^-,\abs{\mu}\in\M(\Omega)$, $\mu=\mu^+-\mu^-$ and $\abs{\mu}=\mu^++\mu^-$,
see \cite[Lemma III.1.6 and Theorem III.1.8]{DS}.

Let us consider the following four subspaces of $\M(\Omega)$:
$$\begin{aligned}
\Ms(\Omega)&=\{\mu\in\M(\Omega)\setsep\mu\mbox{ is $\sigma$-additive}\},\\
\Mp(\Omega)&=\{\mu\in\M(\Omega)\setsep \forall\varepsilon>0\,\exists (B_n)_{n\in\en}\mbox{ a Borel cover of }\Omega: \sum_{n=1}^\infty\abs{\mu}(B_n)<\varepsilon\},\\
\Ma(\Omega)&=\{\mu\in\M(\Omega)\setsep \forall B\in\B(\Omega): \lambda(B)=0\Rightarrow\mu(B)=0\},\\
\Mn(\Omega)&=\{\mu\in\M(\Omega)\setsep \exists B\in\B(\Omega): \lambda(B)=0\ \&\ \abs{\mu}(\Omega\setminus B)=0  \}.
\end{aligned}
$$

The measures from $\Mp(\Omega)$ are called \emph{purely finitely additive}. Let us remark that in \cite[Definition III.7.7]{DS} a different definition of a purely finitely additive measure is given, namely a nonnegative measure is purely finitely additive if it majorizes no nonzero nonnegative $\sigma$-additive measure, and a signed measure is purely finitely additive if it is true both for the positive and the negative parts. It is easy to see that our definition is equivalent (in view of the fact that $\B(\Omega)$ is a $\sigma$-algebra). Moreover, it can be carried to vector measures, which we will do in the next section. 

Further, any $\mu\in \M(\Omega)$ can be canonically decomposed to a $\sigma$-additive and a purely finitely additive measures, see \cite[Theorem III.7.8]{DS}. Indeed, if $\mu\ge 0$ set
\begin{equation*}
S(\mu)(B)=\inf\{\sum_{n\in\en}\mu(C_n)\setsep (C_n)_{n\in\en} \mbox{ is a Borel cover of }B\},\qquad B\in\B(\Omega).
\end{equation*}
It is easy to check that $S(\mu)$ is a nonnegative $\sigma$-additive Borel measure on $\Omega$. Since it is clearly the largest such measure below $\mu$, it is exactly the measure constructed in the proof of \cite[Theorem III.7.8]{DS}.
For a general $\mu$ we set $S(\mu)=S(\mu^+)-S(\mu^-)$. It is easy to show that $S$ is a linear projection and for any $\mu\in\M(\Omega)$ we have
\begin{equation*}
S(\mu)\in\Ms(\Omega),\quad \mu-S(\mu)\in \Mp(\Omega),\quad \norm{\mu}=\norm{S(\mu)}+\norm{\mu-S(\mu)}.
\end{equation*}
Indeed, the first statement is clear, the second one follows from \cite[Theorem 1.17]{YoHe}. By \cite[Theorem 1.24]{YoHe} the decomposition to a $\sigma$-additive and purely finitely additive parts is unique, hence the linearity of $S$ follows. The last equality follows easily using \cite[Theorem 1.16]{YoHe}.

$\sigma$-additive measures from $\Ma(\Omega)$ are exactly the measures which are absolutely continuous with respect to the Lebesgue measure. For finitely additive measures we do not use this term as it has a different meaning, it is defined by a stronger condition \cite[Definition III.4.12]{DS}. The subspace $\Ma(\Omega)$ is important for us, as it is canonically isometric to the dual of $L^\infty(\Omega)$, see \cite[Theorem IV.8.16]{DS}. Moreover, $\Ma(\Omega)$ is weak$^*$-closed in $\M(\Omega)$ and its weak$^*$ topology coincides with the weak$^*$ topology inherited from $\M(\Omega)$. Indeed, $L^\infty(\Omega)$ is the quotient of $\B_b(\Omega,\er)$ by the subspace made by functions which are zero $\lambda$-almost everywhere, so it is enough to use  the standard duality of subspaces and quotients.

Any $\mu\in\M(\Omega)$ can be canonically decomposed to a measure from $\Ma(\Omega)$ and a measure from $\Mn(\Omega)$. Indeed, if $\mu\ge 0$, set
\begin{equation*}
A_s(\mu)(B)=\sup\{\mu(C)\setsep C\subset B\mbox{ Borel}, \lambda(C)=0\},\quad B\in\B(\Omega).
\end{equation*}
It is clear that $A_s(\mu)\in\Mn(\Omega)$. Moreover, if $\lambda(B)=0$, then $A_s(\mu)(B)=\mu(B)$, hence $\mu-A_s(\mu)\in\Ma(\Omega)$. It is easy to check that $A_s$ is additive and positive homogeneous, hence if we define $A_s(\mu)=A_s(\mu^+)-A_s(\mu^-)$ for a general $\mu$, we infer that $A_s$ is a linear projection and that for any $\mu\in\M(\Omega)$ one has
\begin{equation*}
A_s(\mu)\in \Mn(\Omega),\quad \mu-A_s(\mu)\in\Ma(\Omega),\quad \norm{\mu}=\norm{A_s(\mu)}+\norm{\mu-A_s(\mu)}.
\end{equation*}

Further, by \cite[Theorem I.1.10(a)]{hawewe} the projections $S$ and $A_s$ commute.
It follows that for any $\mu\in\M(\Omega)$ we have
\begin{equation*}
\norm{\mu}=\norm{(I-A_s)S\mu}+\norm{A_s S\mu}+ \norm{(I-A_s)(I-S)\mu}+\norm{A_s(I-S)\mu}
\end{equation*}
and
\begin{equation*}
\begin{alignedat}{3}
(I-A_s)S\mu&\in\Ms(\Omega)\cap\Ma(\Omega),&\ 
A_sS\mu&\in\Ms(\Omega)\cap\Mn(\Omega),\\
(I-A_s)(I-S)\mu&\in\Mp(\Omega)\cap\Ma(\Omega),&\quad
A_s(I-S)\mu&\in\Mp(\Omega)\cap\Mn(\Omega).
\end{alignedat}
\end{equation*}

There is another projection of $\M(\Omega)$ onto $\Ms(\Omega)$, different from $S$. Let $\mu\in\M(\Omega)$. Then $\mu$ defines a continuous linear functional on $\B_b(\Omega,\er)$.
The space $\C_0(\Omega)$ (of continuous real-valued functions on the locally compact space $\Omega$ vanishing at infinity) is a closed subspace of $\B_b(\Omega,\er)$, hence we can consider the restriction of the mentioned functional to $\C_0(\Omega)$. By the Riesz representation theorem this functional is represented by a measure from $\Ms(\Omega)$.
Let us denote the resulting measure by $R(\mu)$. It is clear that $R$ is a linear operator and $\norm{R}\le 1$. Further, if $\mu\ge0$, then for any open set $G\subset\Omega$
we have
\begin{equation*}
R(\mu)(G)=\sup\{\mu(K)\setsep K\subset G \mbox{ compact}\}.
\end{equation*}
Indeed, this easily follows from the proof of the Riesz theorem (see \cite[Theorem 2.14]{rudin}), as we have for each open $G\subset\Omega$
$$R(\mu)(G)=\sup\{\mu(f)\setsep f\in\C_0(\Omega), 0\le f\le 1, f \mbox{ supported by a compact subset of }G\}.$$
Hence, if $\mu$ is $\sigma$-additive, then $R(\mu)=\mu$, as $\sigma$-additive finite Borel measures on $\Omega$ are necessarily regular. It follows that $R$ is a norm-one projection of $\M(\Omega)$ onto $\Ms(\Omega)$. This projection will serve as one of the key tools in the proof of the main result.

 Moreover, we have even $R(\Ma(\Omega))=\Ms(\Omega)$, in fact
\begin{equation}\label{eq:R1}
\forall\mu\in\Ms(\Omega)\,\exists\nu\in\Ma(\Omega): R(\nu)=\mu\ \&\ \norm{\nu}=\norm{\mu}.
\end{equation}
Indeed, this follows from the Hahn-Banach theorem, as $\C_0(\Omega)$ is canonically isometric to a subspace of $L^\infty(\Omega)$. In fact, much more is true, as witnessed by the following theorem which shows that the projection $R$ is in a sense `orthogonal' to the projections $A_s$ and $S$. 

\begin{thm}\label{T:R1}
For any $\mu\in\Ms(\Omega)$ there are $\nu_1\in\Mp(\Omega)\cap\Ma(\Omega)$ and $\nu_2\in\Mp(\Omega)\cap\Mn(\Omega)$ with $\norm{\nu_1}=\norm{\nu_2}=\norm{\mu}$ such that $R(\nu_1)=R(\nu_2)=\mu$. In particular,
$R(\Mp(\Omega)\cap \Ma(\Omega))=R(\Mp(\Omega)\cap\Mn(\Omega))=\Ms(\Omega)$. 
\end{thm}

\begin{proof} {\tt Step 1.} Assume that $\mu\ge0$ and the support of $\mu$ is a nowhere dense compact set $K\subset\Omega$.

We start by constructing $\nu_1$. Let $$Y_1=\lin (\C_b(\Omega)\cup\{\bchi_K\}\cup\{\bchi_B\setsep B\in\B(\Omega), \lambda(B)=0\}).$$
Then $Y_1$ is a linear subspace of $\B_b(\Omega,\er)$. Let us define a linear functional $\varphi_1$ on $Y_1$ by setting
$$\varphi_1(f+c\bchi_K+\sum_{i=1}^n\alpha_i\bchi_{B_i})=\int f\di\mu,\quad f\in\C_b(\Omega), B_1,\dots,B_n\in\B(\Omega),\alpha_1,\dots,\alpha_n,c\in\er.$$
Since always
$$\norm{f+c\bchi_K+\sum_{i=1}^n\alpha_i\bchi_{B_i}}_\infty\ge\norm{f}_\infty,$$ it is clear that $\norm{\varphi_1}\le\norm{\mu}$. By the Hahn-Banach extension theorem we can extend $\varphi_1$ to some $\psi_1\in\B_b(\Omega,\er)^*$ with $\norm{\psi_1}=\norm{\varphi_1}$. Let $\nu_1\in\M(\Omega)$ be the measure representing $\psi_1$. By the construction we get $\norm{\nu_1}=\norm{\mu}$ and $R(\nu_1)=\mu$. Since $\nu_1(B)=0$ whenever $\lambda(B)=0$, we see that $\nu_1\in\Ma(\Omega)$. Further, as $\norm{\nu_1}=\norm{\mu}=\mu(\Omega)=\varphi_1(1)=\psi_1(1)=\nu_1(\Omega)$, we deduce that $\nu_1\ge 0$. 
By construction we get $\nu_1(K)=\varphi_1(\bchi_K)=0$. Moreover, since $R(\nu_1)(\Omega\setminus K)=\mu(\Omega\setminus K)=0$, we have $\nu_1(L)=0$ for any compact set $L\subset\Omega\setminus K$. Finally, $\Omega\setminus K$ is $\sigma$-compact, hence $\Omega$ can be covered by a sequence of compact sets of zero $\nu_1$-measure, thus $\nu_1$ is purely finitely additive.

To construct $\nu_2$ we proceed similarly. Set
\begin{multline*}
Y_2=\lin\big(\C_b(\Omega)\cup\{\bchi_K\}\cup\{\bchi_L\setsep L\subset\Omega\mbox{ compact}, \lambda(L)=0\}\\ \cup\{\bchi_F\setsep F\subset\Omega\ \mbox{ is } F_\sigma,\Int F=\emptyset,\lambda(\Omega\setminus F)=0\}\big).\end{multline*}
For $f+c\bchi_K+\sum_{i=1}^n\alpha_i\bchi_{L_i}+\sum_{j=1}^m\beta_j\bchi_{F_j}\in Y_2$ set
$$\varphi_2(f+c\bchi_K+\sum_{i=1}^n\alpha_i\bchi_{L_i}+\sum_{j=1}^m\beta_j\bchi_{F_j})=\int f\di\mu$$ 
It is clear that $\varphi_2$ is a linear functional on $Y_2$. Since $K\cup\bigcup_{i=1}^n L_i\cup \bigcup_{j=1}^m F_j$ has empty interior, we deduce that $\norm{f+c\bchi_K+\sum_{i=1}^n\alpha_i\bchi_{L_i}+\sum_{j=1}^m\beta_j\bchi_{f_j}}_\infty\ge \norm{f}_\infty$
and thus $\norm{\varphi_2}\le\norm{\mu}$. By the Hahn-Banach extension theorem we can extend $\varphi_2$ to some $\psi_2\in\B_b(\Omega,\er)^*$ with $\norm{\psi_2}=\norm{\varphi_2}$. Let $\nu_2\in\M(\Omega)$ be the measure representing $\psi_2$. By the construction we get $\norm{\nu_2}=\norm{\mu}$ and $R(\nu_2)=\mu$. Since $\nu_2(\Omega)=\norm{\nu_2}$, we see that $\nu_2\ge0$. The proof that $\nu_2$ is purely finitely additive is the same as for $\nu_1$. It remains to show that $\nu_2\in\Mn(\Omega)$. To this end fix $G\subset\Omega\setminus K$ dense $G_\delta$ with $\lambda(G)=0$. Then $\nu_2(\Omega\setminus G)=0$, hence indeed $\nu_2\in\Mn(\Omega)$.

{\tt Step 2.} Suppose that $\mu\ge0$.

Let $C\subset\Omega$ be a countable dense set of $\mu$-measure zero. Then $\Omega\setminus C$ has full measure. By regularity of $\mu$ there is a $\sigma$-compact subset $F\subset \Omega\setminus C$ of full $\mu$-measure. Suppose $F=\bigcup_{n\in\en}F_n$ with $F_n$ compact. For each $n\in\en$ the restriction $\mu|_{F_n}$ satisfies the assumption of Step 1, hence there are $\nu_1^n\in\Ma(\Omega)\cap\Mp(\Omega)$ and  $\nu_2^n\in\Mn(\Omega)\cap\Mp(\Omega)$ such that $\norm{\nu_1^n}=\norm{\nu_2^n}=\norm{\mu|_{F_n}}$ and $R(\nu_1^n)=R(\nu_2^n)=\mu|_{F_n}$. It is enough to set
$$\nu_1=\sum_{n\in\en}\nu_1^n,\qquad \nu_2=\sum_{n\in\en}\nu_2^n.$$

{\tt Step 3.} If $\mu$ is a general measure, the proof is completed using the decomposition $\mu=\mu^+-\mu^-$.
\end{proof}

\section{Finitely additive measures -- vector-valued case}\label{s:vector}

We keep the notation from the previous section, in particular, $\Omega$ is a fixed nonempty open subset of $\er^d$, where $d\ge 1$ is given, Moreover, let $X$ be a finite-dimensional normed space. Then $X$ can be represented as $(\er^k,\norm{\cdot}_X)$, where $k=\dim X$. The dual $X^*$ can be then canonically represented as $(\er^k,\norm{\cdot}_{X^*})$.

Consider the following Lebesgue-Bochner spaces
\begin{align*}
L^1(\Omega,X)&=\{\f:\Omega\to X\mbox{ measurable}\setsep \norm{\f}_1=\int \norm{\f(\x)}_X\di\lambda(\x)<\infty \},\\
L^\infty(\Omega,X^*)&=\{\g:\Omega\to X^*\mbox{ measurable}\setsep \norm{\g}_\infty=\esssup_{\x\in\Omega}\norm{\g(\x)}_{X^*}<\infty \}.
\end{align*}
Note that any mapping with values in $X$ or $X^*$ can be represented by a $k$-tuple of
scalar functions. Then $\f=(f_1,\dots,f_k)\in L^1(\Omega,X)$ if and only if $f_1,\dots,f_k\in L^1(\Omega)$. Similarly, $\g=(g_1,\dots,g_k)\in L^\infty(\Omega,X^*)$
if and only if $g_1,\dots,g_k\in L^\infty(\Omega)$. So, the above spaces depend only on the dimension of $X$; the concrete norm on $X$ determines just the norms on these spaces.

It is well-known that $L^\infty(\Omega,X^*)$ is isometric to the dual of $L^1(\Omega,X)$, where the duality is given by
\begin{equation*}
\ip{\g}{\f} = \int_\Omega \ip{\g(\x)}{\f(\x)}\di\lambda(\x),\quad \g\in L^\infty(\Omega,X^*), \f\in L^1(\Omega,X).
\end{equation*}
Indeed, in case $\Omega$ is bounded it follows from \cite[p. 97-98]{DU} (where the assumption that the measure
on $\Omega$ is finite is used). The general case follows then easily (using $\sigma$-finiteness of the Lebesgue measure). Alternatively, one can use \cite[Example 2.19 and the formulas on p.24]{ryan}. Note that the isomorphic version follows easily from the scalar case, but the quoted results are necessary to prove that the identification is in fact isometric.

Further, by $\M(\Omega,X)$ we will denote the space of all $X$-valued finitely additive 
Borel measures on $X$ which have bounded total variation. We equip this space with the total variation norm, i.e.,
\begin{equation}
\norm{\bmu}=\sup \{\sum_{i=1}^n\norm{\bmu(B_i)}_X\setsep B_1,\dots,B_n\in\B(\Omega)\mbox{ disjoint} \},\qquad \bmu\in\M(\Omega,X).
\end{equation}
For any $\bmu\in\M(\Omega,X)$ we define its absolute variation by
\begin{equation}
\abs{\bmu}(B)=\sup \{\sum_{i=1}^n\norm{\bmu(B_i)}_X\setsep B_1,\dots,B_n\subset B\mbox{ disjoint Borel sets} \},\ B\in\B(\Omega).
\end{equation}
We define subspaces $\Ms(\Omega,X)$, $\Ma(\Omega,X)$, $\Mp(\Omega,X)$ and $\Mn(\Omega,X)$ by the same formulas as in the scalar case.

Note that any $\bmu\in\M(\Omega,X)$ can be represented as a $k$-tuple of measures from $\M(\Omega)$, thus $\M(\Omega,X)$ is isomorphic to $\M(\Omega)^k$. Moreover, it is clear that $\bmu=(\mu_1,\dots,\mu_k)\in \Ms(\Omega,X)$ if and only if $\mu_1,\dots,\mu_k\in\Ms(\Omega)$ and similarly for the remaining three subspaces.
So, we can define the projections $\bS:\M(\Omega,X)\to\Ms(\Omega,X)$ and $\bA_s:\M(\Omega,X)\to\Mn(\Omega,X)$
by
\begin{equation}\label{eq:bS,bA}
\begin{array}{c}
\bS(\mu_1,\dots,\mu_k)=(S(\mu_1),\dots,S(\mu_k)), \\ \bA_s(\mu_1,\dots,\mu_k)=(A_s(\mu_1),\dots,A_s(\mu_k)),\end{array}
(\mu_1,\dots,\mu_k)\in\M(\Omega,X).
\end{equation}

Further, $\M(\Omega,X)$ is canonically isometric to the dual of $\B_b(\Omega,X^*)$, the space of bounded $X$-valued Borel mappings on $\Omega$ equipped with the sup-norm. Indeed, simple Borel mappings are dense in $\B_b(\Omega,X^*)$ and it is clear that linear functionals on the space of simple functions are exactly the finitely additive $X$-valued measures and that the norm of such a functional is just the total variation of the vector measure in question. (In fact, this is the proof of the scalar case which works in the vector-valued case as well.)

Similarly, $\Ma(\Omega,X)$ is canonically isometric to the dual of $L^\infty(\Omega,X^*)$.

Further, let $\C_0(\Omega,X^*)$ be the space of continuous $X^*$-valued mappings on $\Omega$ which are zero at infinity equipped with the sup-norm. Its dual $\C_0(\Omega,X^*)^*$ is canonically isometric to $\Ms(\Omega,X)$. Indeed, the isomorphic identification follows easily from the scalar case ($\C_0(\Omega,X^*)$ is isomorphic to $\C_0(\Omega)^k$), so it is enough to observe that this identification is in fact isometric. This can be easily computed similarly as in the scalar case or, alternatively,
it follows by \cite[Theorem 5.22, Theorem 5.33 and the beginning of Section 3.2]{ryan}.

Let us define the mapping $\bR:\M(\Omega,X)\to\Ms(\Omega,X)$ similarly as in the scalar case. I.e., $\bR(\mu)$ is the measure which represents the restriction to $\C_0(\Omega,X^*)$ of the functional on $\B_b(\Omega,X^*)$ represented by $\bmu$. Note that
$$\bR(\mu_1,\dots,\mu_k)=(R(\mu_1),\dots,R(\mu_k)) \mbox{ for } \bmu=(\mu_1,\dots,\mu_k)\in\M(\Omega).$$

The following proposition follows from \eqref{eq:R1} and Theorem~\ref{T:R1}.

\begin{prop}
$$\begin{aligned}
\bR(\Ma(\Omega,X))&=\bR(\Ma(\Omega,X)\cap\Mp(\Omega,X))\\&=\bR(\Mn(\Omega,X)\cap\Mp(\Omega,X))=\Ms(\Omega,X).
\end{aligned}$$
\end{prop}

\section{Lipschitz-free space and its bidual}\label{s:bidual}

In this section we recall the representation of Lipschitz-free spaces on finite-dimensional spaces from \cite{CKK}, we further deduce the representation of its bidual and then prove a key proposition on approximation by test functions (Proposition~\ref{P:aproximace} below).
Let us start with the basic setting. Let $E$ be a normed space of finite dimension $d\ge 2$, $\Omega\subset E$ a nonempty open convex set and $\o\in\Omega$ a distinguished point (we can assume that $\o$ is the origin). We represent $E$ as $(\er^d,\norm{\cdot}_E)$ and $E^*$ as $(\er^d,\norm{\cdot}_{E^*})$.

The following proposition summarizes the results of \cite{CKK}.

\begin{prop}\label{P:popisCKK} \
\begin{itemize}
\item The mapping $T:f\mapsto f'$ is a linear isometry of $\Lip_0(\Omega)$ into $L^\infty(\Omega,E^*)$. The range of $T$ is
$$\begin{aligned}
X(\Omega)&=\{\f=(f_1,\dots,f_d)\in L^\infty(\Omega,E^*)\setsep \forall i,j\in\{1,\dots,d\}: \partial_i f_j=\partial_j f_i\mbox{ in }\Dcp(\Omega)\}
\\&=\{\f\in L^\infty(\Omega,E^*)\setsep \forall i,j\in\{1,\dots,d\}\,\forall\varphi\in\Dc(\Omega): \int_\Omega f_j \partial_i \varphi=\int_\Omega f_i\partial_j \varphi\}.\end{aligned}$$
\item $X(\Omega)=Y(\Omega)^\perp$, where
$$\begin{aligned}
Y(\Omega)&=\{\g\in L^1(\Omega,E)\setsep \div\g=0 \mbox{ in }\Dcp(\er^d)\}
\\&=\{\g=(g_1,\dots,g_d)\in L^1(\Omega,E)\setsep \forall\varphi\in\Dc(\er^d): \sum_{i=1}^d\int_\Omega g_i \partial_i \varphi=0\}.\end{aligned}$$
\item The adjoint mapping $T^*$ maps the quotient $L^1(\Omega,E)/Y(\Omega)$ (which is the canonical predual of $X(\Omega)$)  isometrically onto $\F(\Omega)$.
\end{itemize}
\end{prop}

An easy consequence of the first assertion of the above proposition is the following description of the bidual of $\F(\Omega)$, i.e., of the dual of $\Lip_0(\Omega)$.

\begin{prop}\label{P:bidual}
The adjoint mapping $T^*$ is an isometry of $\Ma(\Omega,E)/X(\Omega)^\perp$ onto $\Lip_0(\Omega)^*$.
\end{prop}

It follows that the question on complementability of $\F(\Omega)$ in its bidual can be reduced to the question on complementability of $L^1(\Omega,E)/Y(\Omega)$ in $\Ma(\Omega,E)/X(\Omega)^\perp$. Note that $L^1(\Omega,E)/Y(\Omega)$ is embedded in $\Ma(\Omega,E)/X(\Omega)^\perp$ in the following way:
Given $\f\in L^1(\Omega,E)$, by $\f\lambda$ we denote the $E$-valued measure on $\Omega$ with density $\f$ with respect to the Lebesgue measure. The mapping $\f\mapsto\f\lambda$ is an isometry of $L^1(\Omega,E)$ onto
$\Ma(\Omega,E)\cap\Ms(\Omega,E)$ (it is onto by the Radon-Nikod\'ym theorem). Hence, we can identify
$L^1(\Omega,E)$ and $\Ma(\Omega,E)\cap\Ms(\Omega,E)$ . Moreover, using this identification, one has $Y(\Omega)=X(\Omega)_\perp=X(\Omega)^\perp\cap \Ma(\Omega,E)\cap\Ms(\Omega,E)$, therefore the inclusion  $L^1(\Omega,E)/Y(\Omega)\subset\Ma(\Omega,E)/X(\Omega)^\perp$
can be obtained by the factorization of the inclusion $L^1(\Omega,E)\subset\Ma(\Omega,E)$ along $Y(\Omega)$, i.e., by the formula
$$\f+Y(\Omega)\mapsto \f\lambda+X(\Omega)^\perp.$$

Note that $L^1(\Omega,E)$ is complemented in its bidual $\Ma(\Omega,E)$ via the projection
$\bS$ defined in \eqref{eq:bS,bA}. In case $E$ is equipped with the $\ell^1$-norm, the projection $\bS$ is even an $L$-projection. A natural idea would be to try to factorize the projection $\bS$ along $X(\Omega)^\perp$. However, it is not possible, as witnessed by Proposition~\ref{P:protipriklady} below.

\smallskip

Now we are going to prove the key proposition on approximation by test functions. Although the representation of
$\F(\Omega)$ and $\F(\Omega)^{**}$ described above works for an arbitrary nonempty open convex subset $\Omega\subset E$, we are able to prove the approximation only for $\Omega=E$. We do not know whether a similar result holds for a general $\Omega$ (cf. Question~\ref{Q:Omega}). Therefore in the sequel we will consider only the case $\Omega=E=(\er^d,\norm{\cdot}_E)$. Instead of $L^1(E,E)$, $L^\infty(E,E^*)$, $\M(E,E)$ etc. we will write $L^1(\er^d,E)$, $L^\infty(\er^d,E^*)$, $\M(\er^d,E)$ etc. This is done in order to stress that the structure of these spaces does not depend on the concrete norm on the domain.

To prove the approximation results, we need the following easy lemma.

\begin{lemma}\label{L:bump} For any $\varepsilon>0$ there exists a $\C^\infty$ function
$\psi:\er^d\to\er$ with the properties
\begin{itemize}
\item $0\le\psi\le 1$ on $\er^d$,
\item $\psi(\x)=1$ whenever $\norm{\x}_E\le 1$,
\item $\psi(\x)=0$ whenever $\norm{\x}_E\ge 2$,
\item $\norm{\nabla\psi}_{L^\infty(\er^d,E^*)}\le 1+\varepsilon$.
\end{itemize}
\end{lemma}

\begin{proof} Fix $\delta\in(0,\frac12)$ such that $\frac{1}{1-2\delta}<1+\varepsilon$.
Define $\psi_0$ by 
$$\psi_0(\x)=\begin{cases}
1 & \mbox{ if }\norm{\x}_E\le 1+\delta,\\
\frac{2-\delta-\norm{\x}_E}{1-2\delta} & \mbox{ if }1+\delta\le\norm{\x}_E\le2-\delta,
\\ 0 & \mbox{ if }\norm{\x}_E\ge 2-\delta.
\end{cases}$$
Then $\psi_0$ is $\frac{1}{1-2\delta}$-Lipschitz, thus $(1+\varepsilon)$-Lipschitz in the norm $\norm{\cdot}_E$.

Further, let $(u_n)$ be an approximate unit in $\Dc(\er^d)$. Fix $n$ so large that $\spt u_n\subset U_E(\o,\delta)$ and set $\psi=u_n*\psi_0$. Then $\psi$ is the sought function.
\end{proof}

\begin{prop}\label{P:aproximace} \ 
\begin{itemize}
\item[(i)] For any $F\in\Lip_0(E)$  and any $\varepsilon>0$ there is a sequence $(F_n)$ in $\Dc(\er^d)$ such that
\begin{itemize}
\item $\norm{F_n}_{lip}\le \norm{F}_{lip}$ for $n\in\en$,
\item the sequence $(F_n-F_n(\o))$ weak$^*$ converges to $F$.
\end{itemize}
\item[(ii)] $\{\nabla \varphi\setsep \varphi\in\Dc(\er^d)\}$ is norm-dense in $\C_0(\er^d,E^*)\cap X(E)$.
\item[(iii)] For any $\g\in\C_0(\er^d,E^*)$ we have
$$\dist(\g,X(E))= \dist(\g,X(E)\cap \C_0(\er^d,E^*)).$$
\end{itemize}
Let us point out that the density in (ii) and the distance in (iii) is considered with respect to the standard norm of $L^\infty(\er^d,E^*)$
as this space serves as the ambient space here.
\end{prop}

\begin{proof} Let $(u_n)$ be an approximate unit in $\Dc(\er^d)$. Without loss of generality suppose that $\spt u_n\subset U_E(\o,1)$ for each $n\in\en$.

(i) Fix $F\in \Lip_0(E)$ and set $L=\norm{F}_{lip}$. For each $n\in\en$ define 
a function $\tilde{F}_{n,0}$ by
$$\tilde{F}_{n,0}(\x)=\begin{cases} F(\x) & \mbox{ if }\norm{\x}_E\le n,\\ 0 & \mbox{ if }\norm{\x}_E\ge 2n.
\end{cases}$$
We claim that $\tilde{F}_{n,0}$ is $L$-Lipschitz on its domain. It is clear that it is $L$-Lipschitz on $B_E(\o,n)$ and on $E\setminus U_E(\o,2n)$. So, fix $\x,\y\in E$ such that $\norm{\x}_E\le n$ and $\norm{\y}_E\ge2n$. Then
$$\abs{\tilde{F}_{n,0}(\x)-\tilde{F}_{n,0}(\y)}=\abs{F(\x)}\le L\cdot\norm{\x}_E\le Ln\le L\norm{\x-\y}_E.$$
Thus we can extend $\tilde{F}_{n,0}$ to an $L$-Lipschitz function $\tilde{F}_n:E\to\er$.
Set $F_n=u_n*\tilde{F}_n$. Then $F_n\in\Dc(\er^d)$, $F_n$ is $L$-Lipschitz,
as
$$\norm{F_n}_{lip}=\norm{\nabla F_n}_{L^\infty(\er^d,E^*)}=\norm{u_n*\tilde{F}'_n}_{L^\infty(\er^d,E^*)}\le\norm{\tilde{F}'_n}_{L^\infty(\er^d,E^*)}=\norm{\tilde F_n}_{lip}=L.$$
So, it remains to show that $(F_n-F_n(\o))$ converges weak$^*$ to $F$. Since the sequence is bounded by $L$ in $\Lip_0(E)$, it is enough to prove the pointwise convergence.
To this end fix $x\in E$. If $n\in\en$ is such that $n\ge\norm{\x}_E+1$, then
$$F_n(\x)=(u_n*\tilde{F}_n)(\x)=(u_n*F)(\x).$$
Since $u_n*F\to F$ pointwise, we deduce that $F_n(\x)\to F(\x)$, thus $F_n\to F$ pointwise.
In particular, $F_n(\o)\to F(\o)=0$, so $F_n-F_n(\o)\to F$ pointwise. This completes the proof of (i).

(ii) Fix $\f\in\C_0(\er^d,E^*)\cap X(E)$. Then $u_n*\f$ uniformly converges to $\f$ and, moreover,
$u_n*\f$ are $\C^\infty$ functions. Moreover, since $\spt u_n\subset U_E(\o,1)$, for each $\x\in\er^d$ we have  (by Lemma~\ref{L:aproxunitodhad})
$$\norm{(u_n*\f)(\x)}_{E^*}\le\sup\{\norm{\f(\y)}\setsep \norm{\y-\x}_E\le 1\},$$
thus $u_n*\f\in\C_0(\er^d,E^*)$. Further, $u_n*\f\in X(E)$, as
\begin{equation}\label{eq:zachovejXE}
\partial_i(u_n*f_{j})=u_n*\partial_if_{j} =u_n*\partial_jf_i=\partial_j(u_n*f_{i})\mbox{ for }i,j\in\{1,\dots,d\}.
\end{equation}
Therefore, we can suppose without loss of generality that $\f$ is moreover $\C^\infty$.

Let $F\in\Lip_0(E)$ be such that $F'=\f$. Note that $F$ is a $\C^\infty$ function.
For $n\in\en$ define $$F_n(\x)=\psi\left(\frac{\x}{n}\right)\cdot F(\x),\quad \x\in E.$$ 
To complete the proof it is enough to show that that the sequence $(\nabla F_n)$ uniformly converges to $\f$. 

For $n\in\en$ set
$$\alpha_n=\sup\{\norm{\f(\x)}_{E^*}\setsep\norm{\x}_E\ge n-1\}.$$
Then  $\alpha_n\searrow 0$. To estimate $\norm{\f-\nabla F_n}_{L^\infty(\er^d,E^*)}$ we distinguish three cases:

\begin{itemize}
\item[(a)] $\norm{\x}_E\le n$. Then $\norm{\f(\x)-\nabla F_n(x)}_{E^*} = 0$.

\item[(b)] $\norm{\x}_E\ge 2n$. Then
$$\norm{\f(\x)-\nabla F_n(x)}_{E^*} = \norm{\f(\x)}_{E^*}\le\alpha_{2n+1}.$$

\item[(c)] $n<\norm{\x}_E<2n$. Then
$$\begin{aligned}
\norm{\f(\x)-\nabla F_n(x)}_{E^*}& = \norm{\f(\x)-\frac1n\nabla\psi\left(\frac{\x}n\right)\cdot F(\x) - \psi\left(\frac{\x}n\right)\cdot \f(\x)}_{E^*}
\\&\le\left(1-\psi\left(\frac{\x}n\right)\right)\cdot\norm{\f(\x)}_{E^*}+\frac{1+\varepsilon}{n}\abs{F(\x)}.
\end{aligned}
$$
It is easy to see that, for each $k\in\en$,
$$\sup_{\norm{\x}_E\le k}\abs{F(\x)}\le \alpha_1+\dots+\alpha_{k}.$$
Hence, we have
$$\norm{\f(\x)-\nabla F_n(x)}_{E^*}\le \alpha_{n+1} + (1+\varepsilon)\frac{ \alpha_1+\dots+\alpha_{2n+1}}{n}.$$
\end{itemize}
\smallskip

Now, summarizing and taking into account that $\alpha_n\to 0$ (and therefore $(\alpha_n)$ is Ces\`aro convergent to $0$), we conclude that $\nabla F_n$ converges uniformly to $\f$, which completes the proof of (ii).

(iii) Fix $\g\in\C_0(\er^d,E^*)$. The inequality `$\le$' is obvious, let us prove the converse one. 
Set $\theta=\dist(\g,X(E))$ and fix any $\varepsilon>0$. Find $\f\in X(E)$ with $\norm{\g-\f}_{L^\infty(\er^d,E^*)}<\theta+\varepsilon$. Since $\g\in\C_0(\er^d,E^*)$, we can find $r>0$ such that 
$\norm{\g(\x)}_{E^*}<\varepsilon$ whenever $\norm{\x}_E\ge r$. Then clearly
\begin{equation}\label{eq:mimokruh}
\norm{\f|_{\er^d\setminus U_E(\o,r)}}_{L^\infty(\er^d\setminus U_E(\o,r),E^*)}<\theta+2\varepsilon.
\end{equation}
Let $F\in\Lip_0(E)$ be such that $F'=\f$. Set $L=\norm{F}_{lip}=\norm{\f}_{L^\infty(\er^d,E^*)}$.
Fix $C>r+1$ and set $B=\{\x\in \er^d\setsep C-1\le \norm{\x}_E\le C\}$.
For $\x,\y\in B$ there are two possibilities:

If the segment $[\x,\y]$ does not meet $U_E(\o,r)$, then $$\abs{F(\x)-F(\y)}\le (\theta+2\varepsilon)\norm{\x-\y}_E.$$

Suppose that the segment $[\x,\y]$ meets $U_E(\o,r)$. This intersection is a segment, let $\u$ and $\vv$ be its endpoints, denoted such that $\vv$ is between $\u$ and $\y$. Then $\norm{\x-\y}_E\geq \norm{\x-\u}_E+\norm{\y-\vv}_E\geq 2(C-1-r)$ and 
$$\begin{aligned}
\abs{F(\x)-F(\y)}&\le \abs{F(\x)-F(\u)}+\abs{F(\u)-F(\vv)}+\abs{F(\vv)-F(\y)}
\\&\le (\theta+2\varepsilon) \norm{\x-\u}_E + L\norm{\u-\vv}_E+  (\theta+2\varepsilon)\norm{\vv-\y}_E
\\&\le  (\theta+2\varepsilon)\norm{\x-\y}_E+2Lr
= \left(\theta+2\varepsilon+\frac{2Lr}{\norm{\x-\y}_E}\right)\norm{\x-\y}_E
\\&\le \left( \theta+2\varepsilon+\frac{Lr}{C-1-r}\right)\norm{\x-\y}_E.\end{aligned}
$$
By choosing $C$ large enough, we can assume that $F$ is $(\theta+3\varepsilon)$-Lipschitz on $B$.

Further, fix $D>C$, set $B_1=B\cup (\er^d\setminus U_E(\o,D))$ and define a function $\tilde{F}_0$ on $B_1$ by
$$\tilde F_0(\x)=\begin{cases}
F(\x)&\mbox{ if }\x\in B,\\ 0 &\mbox{ if }\norm{\x}_E\ge D.
\end{cases}$$
Fix $\x\in B$ and $\y\in\er^d$ with $\norm{\y}_E\ge D$. Then
$$\abs{\tilde{F}_0(\x)-\tilde{F}_0(\y)}=\abs{F(\x)}\le LC\le \frac{LC}{D-C}\norm{\x-\y}_E.
$$
So, if $D$ is large enough, we get that $\tilde{F}_0$ is $(\theta+3\varepsilon)$-Lipschitz on $B_1$. Let $\tilde{F}_1$ be a $(\theta+3\varepsilon)$-Lipschitz extension of $\tilde{F}_0$ to $E\setminus U_E(\o,C-1)$ and 
$$\tilde{F}(\x)=\begin{cases}
F(\x)&\mbox{ if }\norm{\x}_E\le C,\\ \tilde{F}_1(\x), & \mbox{ if }\norm{\x}_E\ge C-1.
\end{cases}$$
Then clearly $\tilde{F}\in\Lip_0(E)$. Set $\tilde{\f}=\tilde{F}'$. Then $\tilde{\f}\in X(E)$ and has compact support. We claim that $$\norm{\tilde{\f}-\g}_{L^\infty(\er^d,E^*)}\le\theta+4\varepsilon.$$
Indeed, on $U_E(\o,C)$ we have $\tilde{\f}=\f$, hence
$$\norm{\tilde{\f}-\g}_{L^\infty(U_E(0,C),E^*)}=\norm{\f-\g}_{L^\infty(U_E(0,C),E^*)}<\theta+\varepsilon.$$
Further, on $\er^d\setminus B_E(\o,C-1)$  the function $\tilde F$ is $(\theta+3\varepsilon)$-Lipschitz, hence
$$\begin{aligned}
\norm{\tilde{\f}-\g}_{L^\infty(\er^d\setminus B_E(0,C-1),E^*)}&\le \norm{\tilde{\f}}_{L^\infty(\er^d\setminus B_E(0,C-1),E^*)}+\norm{\g}_{L^\infty(\er^d\setminus B_E(0,C-1),E^*)}\\&<\theta+4\varepsilon.\end{aligned}$$

Finally, consider the functions $u_n*\tilde{\f}$. They belong to $C_0(\er^d,E^*)$ (in fact, even to $\Dc(\er^d,E^*)$) and also to $X(E)$ (by the computation \eqref{eq:zachovejXE} above). Further
$$\norm{u_n*\tilde{\f}-\g}_{L^\infty(\er^d,E^*)}\le \norm{u_n*(\tilde{\f}-\g)}_{L^\infty(\er^d,E^*)}
+\norm{u_n*\g-\g}_{L^\infty(\er^d,E^*)}<\theta+5\varepsilon$$
if $n$ is large enough, as the first term is bounded by $\theta+4\varepsilon$ by Lemma~\ref{L:aproxunitodhad} and the second one goes to zero (because $u_n*\g$ converges to $\g$ uniformly by Lemma~\ref{L:aproximateunitconvergence}).

Since $\varepsilon>0$ was arbitrary, this completes the proof.
\end{proof}

\section{On divergence of measures}\label{s:divergence}

In this section we investigate the space $\Ms(\er^d,E)$ and its subspaces
$$\begin{aligned}
Y_\sigma(E)&=\{\bmu\in\Ms(\er^d,E)\setsep\div\bmu=0\mbox{ in }\Dcp(\er^d)\},\\
\Ms^{div}(E)&=\{\bmu\in\Ms(\er^d,E)\setsep\exists\f\in L^1(\er^d,E): \div\f=\div\bmu\}
\\& =(\Ms(\er^d,E)\cap\Ma(\er^d,E))+Y_\sigma(E).
\end{aligned}$$
The aim of this section is to prove three results. The first one is a characterization of the subspace $\Ms^{div}(E)$ in Proposition~\ref{P:spojitost}.
Similar results in a different context can be found in \cite[Proposition 8.12]{maly14} or \cite[Lemma 2.1]{boChaJi}. The second result is the canonical isomorphism of $L^1(\er^d,E)/Y(E)$ and $\Ms^{div}(E)/Y_\sigma(E)$ (see Lemma~\ref{l:isomorfismusJ}). The last result is Proposition~\ref{p:podminka} where a stability property of $\Ms^{div}(E)$ is proved. It will be used in the next section.

We start by the following extension lemma

\begin{lemma}\label{L:rozsirovanizY} For any 
$\bmu\in Y_\sigma(E)$ there is some $\bnu\in X(E)^\perp$ such that $\bR(\bnu)=\bmu$ and $\norm{\bnu}\le \norm{\bmu}$.
\end{lemma}

\begin{proof} Note that $\bmu$ defines a functional on $\C_0(\er^d,E^*)$. We will extend it to a functional from $X(E)^\perp$. To this end we define the functional $\varphi$ on $\C_0(\er^d,E^*)+X(E)$ by the formula
$$\varphi(\f+\g)=\ip{\bmu}{\f}=\sum_{i=1}^d \int f_i\di\mu_i,\quad \f\in\C_0(\er^d,E^*), \g\in X(E).$$
First observe that $\varphi$ is well defined. Indeed, the assumption $\div\bmu=0$ means that $\bmu$ is zero on the space $\{\nabla\varphi\setsep\varphi\in\Dc(\er^d)\}$. By Proposition~\ref{P:aproximace}(ii) we get that $\bmu$ vanishes also on $\C_0(\er^d,E^*)\cap X(E)$.

Further, let us prove that $\norm{\varphi}\le \norm{\bmu}$. Let $\varepsilon>0$ be arbitrary. Let $\f\in\C_0(\er^d,E^*)$ and $\g\in X(E)$ be such that $\norm{\f+\g}\le 1$. Hence $\dist(\f,X(E))\le 1$. By Proposition~\ref{P:aproximace}(iii) there is $\h\in X(E)\cap\C_0(\er^d,E^*)$ such that $\norm{\f-\h}_{L^\infty(\er^d,E^*)}<1+\varepsilon$. Then
$$\abs{\varphi(\f+\g)}=\abs{\varphi(\f-\h+\g)}=\abs{\ip{\bmu}{\f-\h}}\le\norm{\bmu}\cdot(1+\varepsilon).$$
We deduce that $\norm{\varphi}\le (1+\varepsilon)\norm{\bmu}$. Since $\varepsilon>0$ was arbitrary, $\norm{\varphi}\le \norm{\bmu}$. 

Finally, by the Hahn-Banach extension theorem the functional $\varphi$ can be extended to a functional on $L^\infty(\er^d,E^*)$ with same norm. Its representing measure is the sought $\bnu$.
\end{proof}

As a consequence we get the following proposition.

\begin{prop}\label{p:projekceR} $\bR(X(E)^\perp)=Y_\sigma(E)$.
\end{prop}

\begin{proof}
The inclusion $\subset$ is obvious, the converse one follows from Lemma~\ref{L:rozsirovanizY}.
\end{proof}

To prove the key Proposition~\ref{P:spojitost} we will use the following abstract lemma.

\begin{lemma}\label{L:rozsirovaniabstraktni}
Let $Z$ be a separable Banach space and let $V\subset Z^*$ be a $C$-norming subspace (where $C\ge1$), i.e., $B_{Z^*}\subset \wscl{V\cap C B_{Z^*}}$.
Let $\theta$ be a bounded linear functional on $V$. The following assertions are equivalent.
\begin{itemize}
\item[(i)] There is $z\in Z$ such that $\theta(z^*)=z^*(z)$ for $z^*\in V$.
\item[(ii)] $\theta$ is weak$^*$-sequentially continuous on $V$.
\end{itemize}
Moreover, if these conditions are satisfied, then $z$ is uniquely determined and $\norm{z}\le C\norm{\theta}$.
\end{lemma}

\begin{proof}
The implication (i)$\Rightarrow$(ii) is obvious, let us prove the converse one.
Since $Z$ is separable, there is a translation invariant metric $\rho$ on $Z^*$ which 
generates the weak$^*$ topology on any bounded set. Hence, if $\theta$ is weak$^*$-sequentially continuous on $V$, it is weak$^*$-continuous on $V\cap r B_{Z^*}$
for each $r>0$. We claim that it is even uniformly continuous (in the metric $\rho$) on $V\cap r B_{Z^*}$. 

Indeed, fix $r>0$ and $\varepsilon>0$. By the weak$^*$-continuity on $V\cap 2r B_{Z^*}$ there is $\delta>0$ such that $\abs{\theta(z^*)}<\varepsilon$ whenever $z^*\in V\cap 2r B_{Z^*}$ and $\rho(z^*,0)<\delta$. Then, whenever $z_1^*,z_2^*\in V\cap r B_{Z^*}$ are such that $\rho(z_1^*,z_2^*)<\delta$, we have $z_1^*-z_2^*\in V\cap 2r B_{Z^*}$ and $\rho(z_1^*-z_2^*,0)<\delta$, hence $\abs{\theta(z_1^*)-\theta(z_2^*)}<\varepsilon$.

It follows that $\theta$ can be uniquely extended to a functional $\tilde{\theta}$ on $Z^*$ such that $\tilde{\theta}$ is weak$^*$-continuous on $\wscl{V\cap r B_{Z^*}}$ for any $r>0$. It is clear that $\tilde{\theta}$ is linear. It remains to show that it is represented by an element of $Z$ and to estimate its norm.

To this end we use the assumption that $V$ is $C$-norming. Since $$B_{Z^*}\subset\wscl{V\cap C B_{Z^*}},$$ we deduce that $\tilde{\theta}$ is weak$^*$ continuous on $B_{Z^*}$, thus it belongs to $Z$ by the Banach-Dieudonn\'e theorem. Moreover,
$$\norm{\tilde{\theta}}=\sup\left\{\abs{\tilde{\theta}(z^*)}\setsep z^*\in B_{Z^*}\right\}
\le \sup\{\abs{{\theta}(v^*)}\setsep v^*\in V\cap C B_{Z^*}\} = C\norm{\theta}.$$
This completes the proof.
\end{proof}

\begin{prop}\label{P:spojitost}\
\begin{itemize}
\item[(a)] A measure $\bmu\in\Ms(\er^d,E)$ belongs to $\Ms^{div}(E)$ if and only if the following condition holds:

\noindent Given a sequence $(\varphi_n)$ in $\Dc(\er^d)$ satisfying
\begin{itemize}
\item[$\circ$] $\varphi_n-\varphi_n(\o)$ pointwise converge to zero,
\item[$\circ$] $\sup_{n\in\en}\norm{\nabla\varphi_n}_{L^\infty(\er^d,E^*)}<\infty$;
\end{itemize}
one has $\ip{\bmu}{\nabla\varphi_n}\to0$.
\item[(b)] For any $\bmu\in\Ms^{div}(E)$ and any $\varepsilon>0$ there is $\f\in L^1(\er^d,E)$ with $\div\bmu=\div\f$ and $\norm{\f}\le(1+\varepsilon)\norm{\bmu}$.
\end{itemize}
\end{prop}

\begin{proof} This is a consequence of Lemma~\ref{L:rozsirovaniabstraktni} and Proposition~\ref{P:popisCKK}. Let us explain it:

We will apply Lemma~\ref{L:rozsirovaniabstraktni} to $Z=\F(E)$ (hence $Z^*=\Lip_0(E)$) and $$V=\{\varphi-\varphi(\o)\setsep \varphi\in\Dc(\er^d)\}.$$ By Proposition~\ref{P:aproximace}(i) $V$ is $1$-norming.

Any $\bmu\in\Ms(\er^d,E)$ generates a functional on $V$ by $F_{\bmu}:\varphi\mapsto \ip{\bmu}{\nabla\varphi}$. The two conditions in (a) then mean that $F_{\bmu}$ is weak$^*$-sequentially continuous on $V$. So, by Lemma~\ref{L:rozsirovaniabstraktni} the validity of the two conditions is equivalent to the existence of $\theta\in \F(E)$ such that $F_{\bmu}=\theta|_V$. By Proposition~\ref{P:popisCKK} $\F(E)$ is canonically isometric to $L^1(\er^d,E)/Y(E)$. If $\theta\in \F(E)$ and $[\g]=(T^*)^{-1}(\theta)\in L^1(\er^d,E)/Y(E)$, then
$$\ip{\theta}{\varphi-\varphi(\o)}=\ip{T^*[\g]}{\varphi-\varphi(\o)}=\ip{[\g]}{T(\varphi-\varphi(\o)}=\ip{\g}{\nabla\varphi}, \quad \varphi\in \Dc(\er^d).$$
Therefore, the existence of $\theta\in\F(E)$ with $F_{\bmu}=\theta|_V$ is equivalent to the existence of $\g\in L^1(\er^d,E)$ with $F_{\g\lambda}=F_{\bmu}$. Finally, the equality $F_{\g\lambda}=F_{\bmu}$ means exactly that $\div\g=\div\bmu$ in $\Dcp(\er^d)$. This completes the proof of (a).

To show (b) suppose that $\bmu\in\Ms^{div}(E)$. Fix $\g\in L^1(\er^d,E)$ with $\div\g=\div\bmu$. Then $T^*([\g])|_V=F_{\bmu}$ (where $F_{\bmu}$ is defined as above). Therefore, using the isometric identification from Proposition~\ref{P:popisCKK} and Lemma~\ref{L:rozsirovaniabstraktni} we get
$$\norm{[\g]}=\norm{T^*([\g])}\le \norm{F_{\bmu}}\le \norm{\bmu}.$$
Hence, by the definition of the quotient norm, given $\varepsilon>0$, there
is $\f\in L^1(\er^d,E)$ with $\f-\g\in Y(E)$ such that $\norm{\f}\le(1+\varepsilon)\norm{\bmu}$. Then $\div\f=\div\g=\div\bmu$.
\end{proof}

\begin{lemma}\label{l:isomorfismusJ} The subspace $\Ms^{div}(E)$ is norm-closed in $\Ms(\er^d,E)$. Moreover,
the canonical identity mapping $J: L^1(\er^d,E)/Y(E)\to \Ms^{div}(E)/Y_\sigma(E)$ defined by
$$J(\f+Y(E))=\f\lambda+Y_\sigma(E),\quad [\f]=\f+Y(E)\in L^1(\er^d,E)/Y(E),$$
is an onto isometry.
\end{lemma}

\begin{proof} To prove $\Ms^{\div}(E)$ is norm-closed it is enough to prove that $\sum_{n=1}^\infty\bmu_n\in \Ms^{\div}(E)$ whenever $(\bmu_n)$ is a sequence in $\Ms^{div}(E)$ such that $\sum_{n=1}^\infty\norm{\bmu_n}<\infty$. So, let $(\bmu_n)$ be such a sequence. Let $\bmu\in\Ms(\er^d,E)$ be its sum (by completeness of $\Ms(\er^d,E)$ it exists). By Proposition~\ref{P:spojitost}(b) there is a sequence $(\f_n)$ in $L^1(\er^d,E)$ such that $\div\bmu_n=\div\f_n$ and $\norm{\f_n}\le 2\norm{\bmu_n}$ for each $n\in\en$. Therefore $\sum_{n=1}^\infty\norm{\f_n}<\infty$, so by completeness of $L^1(\er^d,E)$ we have $\f=\sum_{n=1}^\infty \f_n\in L^1(\er^d,E)$. Moreover, clearly $\div\f=\div\bmu$.

Let us continue by proving the moreover part.
$J$ is well defined as $Y(E)\subset Y_\sigma(E)$ (more precisely, as $\f\lambda\in Y_\sigma(E)$ whenever $\f\in Y(E)$). The same inclusion shows that $\norm{J}\le1$. Further, $J$ is one-to-one, as for $f\in L^1(\er^d,E)$, $\f\lambda\in Y_\sigma(E)$ is equivalent to $\f\in Y(E)$. 

Let $[\bmu]=\bmu+Y_\sigma(E)\in \Ms^{\div}(E)/Y_\sigma(E)$. Since $\bmu\in\Ms^{\div}(E)$, there is $\f\in L^1(\er^d,E)$ with $\div\f=\div\bmu$. Then $\bmu-\f\lambda\in Y_\sigma(E)$, thus $[\bmu]=J([\f])$. It follows that $J$ is onto. Moreover, by Proposition~\ref{P:spojitost} we get $\norm{[\f]}\le \norm{\bmu}$. This proves that $\norm{J^{-1}}\le 1$.
\end{proof}

The next proposition is a stability result on $\Ms^{div}(E)$.

\begin{prop}\label{p:podminka} 
For any $\bmu\in\Ms^{\div}(E)$ and any $g\in L^1(\abs{\bmu})$, the measure $g\bmu$ defined by
$$g\bmu(B)=\left(\int_B g\di\mu_1,\dots,\int_B g\di\mu_d\right),\quad B\in\B(E)$$
belongs to $\Ms^{div}(E)$ as well.
\end{prop}

\begin{proof} Suppose that $\bmu\in \Ms^{div}(E)$, $g\in L^1(\abs{\bmu})$ and $g\bmu\notin\Ms^{div}(E)$. By Proposition~\ref{P:spojitost} there is a sequence $(\varphi_n)$ in $\Dc(\er^d)$ and $C>0$ such that
\begin{itemize}
\item $\varphi_n-\varphi_n(\o)$ pointwise converge to zero,
\item $\norm{\nabla\varphi_n}_{L^\infty(\er^d,E^*)}\le 1$ for $n\in\en$,
\item $\ip{g\bmu}{\nabla\varphi_n}>C$ for $n\in\en$.
\end{itemize}
By the Radon-Nikod\'ym theorem there is a function $\f\in L^1(\abs{\bmu},E)$ such that $\mu_i= f_i\abs{\bmu}$ for $i=1,\dots,d$. Moreover, we have $\norm{\f}_{L^\infty(\abs{\bmu},E)}\leq 1$. Indeed, in order to get a contradiction, let us assume that the set $A:=\{\x\setsep \|\f(\x)\| > 1\}$ is of positive $\abs{\bmu}$-measure. Fixing a countable dense set $\{\y_n\setsep n\in\en\}\subset S_{E^*}$, we see that $A = \bigcup_{n\in\en}\{\x\setsep \ip{\y_n}{\f(\x)} > 1\}$; hence, there is $n\in\en$ such that the set $A_n:=\{\x\setsep \ip{\y_n}{\f(\x)} > 1\}$ is of positive $\abs{\bmu}$-measure. Then we have
$$\abs{\bmu}(A_n) < \int_{A_n} \ip{\y_n}{\f(\x)}\di\abs{\bmu} = \ip{\y_n}{\int_{A_n} \f(\x)\di\abs{\bmu}} = \ip{\y_n}{\bmu(A_n)} \leq \norm{\bmu(A_n)}_E,$$
a contradiction.

Observe that $\Dc(\er^d)$ are norm dense in $L^1(\abs{\bmu})$. Indeed, given $h\in L^1(\abs{\bmu})$ and $\varepsilon>0$, the Luzin theorem yields $h_1$, a continuous function with compact support such that $\norm{h-h_1}_{L^1(\abs{\bmu})}<\varepsilon$. Further, there is $h_2\in\Dc(\er^d)$ with $\norm{h_1-h_2}_\infty<\varepsilon$. Then $\norm{h_1-h_2}_{L^1(\abs{\bmu})}<\varepsilon\norm{\bmu}$, so  $\norm{h-h_2}_{L^1(\abs{\bmu})}<\varepsilon(1+\norm{\bmu})$.

Let $h\in L^1(\abs{\bmu})$ be any function. Then
$$\begin{aligned}
\abs{\ip{g\bmu}{\nabla\varphi_n}-\ip{h\bmu}{\nabla\varphi_n}}
&=\abs{\sum_{i=1}^d \int (g(\x)-h(\x))\partial_i\varphi_n(\x) f_i(\x)\di\abs{\bmu}(\x)}
\\&\le \int \abs{g(\x)-h(\x)}\norm{\nabla\varphi_n(\x)}_{E^*}\norm{\f(\x)}_E\di\abs{\bmu}(\x)\\&
\le \norm{g-h}_{L^1(\abs{\bmu})}.\end{aligned}$$
Therefore, we can find $h\in\Dc(\er^d)$ (with  $\norm{g-h}_{L^1(\abs{\bmu})}$ small enough) such that  $\ip{h\bmu}{\nabla\varphi_n}>\frac C2$ for $n\in\en$. 
Then, for any $n\in\en$, we have
$$\begin{aligned}
\frac{C}{2}&<\ip{h\bmu}{\nabla\varphi_n}=\ip{h\bmu}{\nabla(\varphi_n-\varphi_n(\o))}
=\ip{\bmu}{h\nabla(\varphi_n-\varphi_n(\o))}
\\&=\ip{\bmu}{\nabla(h(\varphi_n-\varphi_n(\o))} - \ip{\bmu}{(\varphi_n-\varphi_n(\o))\nabla h}\end{aligned}$$
We will show that the last expression goes to zero as $n\to\infty$, which will be a contradiction. To this end first observe that
\begin{equation}\begin{aligned}
\norm{(\varphi_n-\varphi_n(\o))\nabla h}_{L^\infty(\er^d,E^*)}
&\le\norm{\nabla h}_{L^\infty(\er^d,E^*)}\cdot\sup_{\x\in\spt h}\norm{\varphi_n(\x)-\varphi_n(\o)}\\&\le\norm{\nabla h}_{L^\infty(\er^d,E^*)}
\cdot\sup_{\x\in\spt h}\norm{\x},\end{aligned}
\end{equation}
so the sequence $((\varphi_n-\varphi_n(\o))\nabla h)$ is uniformly bounded.
By Lebesgue theorem it follows that
$\ip{\bmu}{(\varphi_n-\varphi_n(\o))\nabla h}\to0.$

The remaining term $\ip{\bmu}{\nabla(h(\varphi_n-\varphi_n(\o))}$ converges to zero as well by Proposition~\ref{P:spojitost}. Indeed, $\bmu\in\Ms^{\div}(E)$ and $(h(\varphi_n-\varphi_n(\o)))$ is a sequence of test functions on $\er^d$ vanishing at $\o$, pointwise converging to zero and their gradients are uniformly bounded, as 
$$\begin{aligned}
\norm{\nabla(h(\varphi_n-\varphi_n(\o)))}_{L^\infty(\er^d,E^*)}
&\le \norm{h\cdot\nabla\varphi_n}_{L^\infty(\er^d,E^*)}+\norm{\nabla{h}\cdot(\varphi_n-\varphi_n(\o))}_{L^\infty(\er^d,E^*)}
\\& \le \norm{h}_\infty+\norm{\nabla h}_{L^\infty(\er^d,E^*)}\cdot\sup_{\x\in\spt h}\norm{\x}.\end{aligned}$$
This completes the proof.
\end{proof}

\section{Tangent spaces of a measure and a projection onto $\Ms^{div}(E)$}\label{s:tecny}

The aim of this section is to show that $\Ms^{div}(E)$ is a complemented subspace of $\Ms(\er^d,E)$. The projection is defined in \eqref{eq:projekce} and a proof that it is indeed
a projection with required properties is given in Proposition~\ref{P:projekcenaMsdiv}. This result will be used in the next section to prove our main result, Theorem~\ref{T:main}.

The basic tool will be the notion of tangent spaces to a given nonnegative measure in the sense of \cite{hiUm}. This approach was used to obtain similar results in a different context for example in \cite{boChaJi}.

Recall that $E=(\er^d,\norm{\cdot}_E)$ is a given $d$-dimensional normed space (where $d\ge2$). Further, let $C=d_{BM}(E,\ell^2_d)$ be the Banach-Mazur distance of $E$ and the $d$-dimensional Hilbert space. Then there is a Euclidean norm $\norm{\cdot}_2$ on $E$ (i.e., a norm induced by an inner product) such that $\norm{\cdot}_2\le\norm{\cdot}_E\le C\norm{\cdot}_2$. Let us fix such a norm and, given a linear subspace $F\subset E$, denote by $P_F$ the orthogonal projection of $E$ onto $F$ with respect to the norm $\norm{\cdot}_2$. Below we use $\dist_2$ to denote distance in the norm $\norm{\cdot}_2$ and  $U_2(\x,r)$ to denote open balls in this norm.

Further, given a nonnegative measure $\nu\in\Ms(\er^d)$, we set
\begin{equation*}
 \M_\nu(E)=\{\f\in L^1(\nu,E)\setsep \f\nu\in\Ms^{div}(E)\}.
\end{equation*}

The following proposition provides a characterization of $M_\nu$ analogous to \cite[Proposition 3.2(i)]{boChaJi}. It will be crucial to complete the proof, together with the results of \cite{hiUm}.

\begin{prop}\label{P:Tnu}Let $\nu\in\Ms(E)$ be a nonnegative measure. 
    Then there exists a mapping $T_\nu$ assigning to each $\x\in \er^d$ a vector subspace $T_\nu(\x)\subset E$ with the following properties:
    \begin{itemize}
    \item $T_\nu$ is lower $\nu$-measurable, i.e. $\{\x\in \er^d\setsep T_\nu(\x)\cap G\ne\emptyset\}$ is $\nu$-measurable for any $G\subset E$ open.
    \item For any $\f\in L^1(\nu,E)$ we have
    $$\f\in \M_\nu(E) \Leftrightarrow \f(\x)\in T_\nu(\x)\mbox{ for $\nu$-almost all }\x\in \er^d.$$
    \end{itemize}
  Moreover, there exists a sequence $(\f_n)_{n=1}^\infty$ of functions from $\M_\nu(E)$  such that 
$$T_\nu(\x) = \overline{\{\f_n(\x)\setsep n\in\en\}}\mbox{ for $\nu$-almost all } \x\in \er^d.$$
\end{prop}

\begin{proof} By Lemma~\ref{l:isomorfismusJ} we know that $\M_\nu(E)$ is a closed subspace of $L^1(\nu,E)$. Further, it follows from Proposition~\ref{p:podminka} that $\bchi_B\cdot\f\in \M_\nu(E)$ whenever $\f\in \M_\nu(E)$ and $B\subset E$ is $\nu$-measurable. It follows 
from \cite[Theorem 3.1]{hiUm} that there is a mapping $T_\nu$ assigning to each $\x\in \er^d$ a nonempty closed subset $T_\nu(\x)\subset E$ such that the two conditions are satisfied. Moreover, by 
\cite[Lemma 1.1]{hiUm} there is a sequence $(\f_n)$ in $\M_\nu(E)$ such that $T_\nu(\x)$ is given by the above formula. Since $\M_\nu(E)$ is a linear subspace, the sequence $(\f_n)$ can be extended to one generating a  $\qe$-linear subspace, so we can deduce that $T_\nu(\x)$ is a linear subspace for $\nu$-almost all $\x\in \er^d$. Since the change of $T_\nu$ on a set of $\nu$-measure zero does not affect the statement, we can suppose that $T_\nu(\x)$ is a linear subspaces for each $\x\in E$.
\end{proof}

In the next proposition we construct a projection of $L^1(\nu,E)$ onto $\M_\nu(E)$. A similar idea was used in \cite[Proposition 3.2(ii)]{boChaJi}

\begin{prop}\label{P:projekcenaL1nu}
Let $\nu\in\Ms(E)$ be a nonnegative measure and $T_\nu$ be the mapping provided by Proposition~\ref{P:Tnu}. For any $\f\in L^1(\nu,E)$ define
$$P_\nu(\f)(\x)=P_{T_\nu(\x)}(\f(\x)),\quad \x\in E.$$
Then $P_\nu$ is a linear projection of $L^1(\nu,E)$ onto $\M_\nu(E)$ with $\norm{P_\nu}\le C$.
\end{prop}

\begin{proof} Let us start by proving that $P_\nu(\f)$ is a $\nu$-measurable function for each $\f\in L^1(\nu,E)$. To this end recall that the orthogonal projection coincides with the metric projection, i.e., the nearest-point mapping. Fix an open set $G\subset E$. Let $D$ be a countable dense subset of $E$. 

The first step is to show that the mapping
$$d(\x)=\dist_2(\f(\x),T_\nu(\x)),\quad \x\in E$$
is $\nu$-measurable. This follows immediately from the following equivalence which holds for each $r>0$.
\begin{equation}
d(\x)<r \Leftrightarrow \exists \y\in D: \f(\x)\in U_2\left(\y,\frac r2\right)\ \&\ U_2\left(\y,\frac r2\right)\cap T_\nu(\x)\ne\emptyset.
\end{equation}
The implication $\Leftarrow$ follows from the triangle inequality. Let us show the converse one.
Suppose $d(\x)<r$. Then there is $\z\in T_\nu(\x)$ with $\norm{\f(\x)-\z}_2<r$. Set $\y_0=\frac12(\f(\x)+\z)$. Then  $\norm{\f(\x)-\y_0}_2=\norm{\y_0-\z}_2<\frac r2$. Thus we can find $\y\in D$ close enough to $\y_0$ such that  $\norm{\f(\x)-\y}_2<\frac r2$ and $\norm{\y-\z}_2<\frac r2$.
This $\y$ witnesses the validity of the right-hand side.

Further, given any $\y\in E$ and $r>0$, we will show that
\begin{equation} \label{eq:meritelnost}
\begin{aligned}
P_\nu(\f)(\x)\in U_2(\y,r) \Leftrightarrow &\ \exists \alpha,\beta\in\qe\,\exists n\in\en\,\exists \u\in D: 
4\beta^2-\alpha^2<\frac1{n^2}\ \&\  \\ &\ d(\x)>\alpha\ \&\ \f(\x)\in U_2(\u,\beta)\ \&\ \\&\ T_\nu(\x)\cap U_2(\u,\beta)\cap U_2\left(\y,r-\frac1n\right)\ne\emptyset.
\end{aligned}
\end{equation}
Once this is proved, it is clear that $P_\nu(\f)^{-1}(U_2(\y,r))$ is a $\nu$-measurable set for any $\y\in E$ and $r>0$, hence the mapping  $P_\nu(\f)$ is $\nu$-measurable. So, let us prove the equivalence.

$\Rightarrow$: Let $\z=P_\nu(\f)(\x)\in U_2(\y,r)$. 
Fix $n\in\en$ such that $\norm{\z-\y}_2<r-\frac1n$. Further, fix $\alpha,\beta\in\qe$ such that
$\alpha<d(\x)<2\beta$ and $4\beta^2-\alpha^2<\frac1{n^2}$. Set $\u_0=\frac12(\f(\x)+\z)$. Then
 $\norm{\f(\x)-\u_0}_2=\norm{\u_0-\z}_2<\beta$. So, we can find $\u\in D$ close enough to $\u_0$ such that  $\norm{\f(\x)-\u}_2<\beta$ and $\norm{\u-\z}_2<\beta$. It is clear that $\alpha,\beta,n,\u$ witness the validity of the right-hand side.

$\Leftarrow$: Choose the respective $\alpha,\beta,n,\u$. Fix $\z_0\in  T_\nu(\x)\cap U_2(\u,\beta)\cap U_2(\y,r-\frac1n)$ and set $\z=P_\nu(\f(\x))$. Then
$$\begin{aligned}
\norm{\z-\z_0}_2^2&=\norm{\f(\x)-\z_0}_2^2-\norm{\f(\x)-\z}_2^2
\le (\norm{\f(x)-\u}_2+\norm{\z_0-\u}_2)^2-d(\x)^2\\ &\le 4\beta^2-\alpha^2<\frac1{n^2}.
\end{aligned}$$
It follows that 
$$\norm{\z-\y}_2\le\norm{\z-\z_0}_2+\norm{\z_0-\y}_2<r,$$
thus $\z\in U_2(\y,r)$.

\smallskip

So, we have proved the equivalence \eqref{eq:meritelnost} and hence we know that $P_\nu(\f)$ is $\nu$-measurable for each $\f\in L^1(\nu,E)$. We continue by estimating its norm.
Since $\norm{P_\nu(\f)(\x)}_2\le \norm{\f(\x)}_2$ for each $\x\in \er^d$, we get
$$\begin{aligned}
\norm{P_\nu(\f)}_{L^1(\nu,E)}&=\int\norm{P_\nu(\f)(\x)}_E\di\nu(\x)
\le C\int\norm{P_\nu(\f)(\x)}_2\di\nu(\x)
\\&\le C\int\norm{\f(\x)}_2\di \nu(\x) 
\le C\int\norm{\f(\x)}_E\di \nu(\x)=\norm{\f}_{L^1(\nu,E)},\end{aligned}$$
thus $P_\nu$ maps $L^1(\nu,E)$ into $L^1(\nu,E)$. Moreover, it is clear that
$P_\nu$ is linear and the above estimate shows that $\norm{P_\nu}\le C$.
Finally, it follows from Proposition~\ref{P:Tnu} that $P_\nu$ is a projection with range $\M_\nu(E)$.
\end{proof}

The previous proposition describes a projection of $L^1(\nu,E)$ onto $\M_\nu(E)$. Next we will glue these projections to get a projection of $\Ms(\er^d,E)$ onto $\Ms^{div}(E)$. The projection will be defined by
setting
\begin{equation}\label{eq:projekce}
P(\bmu)= P_\nu(\f)\nu\quad\mbox{if }\bmu=\f\nu\mbox{ where }\nu\in\Ms(E), \nu\ge0, \f\in L^1(\nu,E).
\end{equation}
Note that any $\bmu\in \Ms(\er^d,E)$ can be, due to the Radon-Nikod\'ym theorem, expressed as $\bmu=\f\abs{\bmu}$ where $\f\in L^1(\abs{\bmu},E)$ (in fact, $\f\in L^\infty(\abs{\bmu},E)$, see the proof of Proposition~\ref{p:podminka}).
We will show that $P$ is a well-defined linear projection of norm at most $C$. To this end we need the following lemma.

\begin{lemma}\label{L:rovnostae}
Let $\mu,\nu\in\Ms(E)$ be nonnegative measures such that $\mu$ is absolutely continuous with respect to $\nu$. Then $T_\nu(\x)=T_\mu(\x)$ for $\mu$-almost all $\x\in \er^d$.
\end{lemma}

\begin{proof}
By the Radon-Nikod\'ym theorem there is $g\in L^1(\nu)$ such that $g\ge 0$ and $\mu=g\nu$.
Fix any $\f\in L^1(\mu,E)$. We will show 
\begin{equation}\label{eq:munu}
\f(\x)\in T_\mu(\x)  \mbox{ for $\mu$-almost all }\x\in \er^d
\Leftrightarrow \f(\x)\in T_\nu(\x)  \mbox{ for $\mu$-almost all }\x\in \er^d
\end{equation}
This can be proved by the following sequence of equivalences:
$$\begin{aligned}
\f(\x)\in T_\mu(\x) & \mbox{ for $\mu$-almost all }\x\in \er^d 
\Leftrightarrow \f\mu\in\Ms^{div}(E)
\Leftrightarrow \f g\nu\in\Ms^{div}(E)
\\ &\Leftrightarrow  \f(\x)g(\x)\in T_\nu(\x)  \mbox{ for $\nu$-almost all }\x\in \er^d
\\ &\Leftrightarrow  \f(\x)g(\x)\in T_\nu(\x)  \mbox{ for $\mu$-almost all }\x\in \er^d
\\& \Leftrightarrow  \f(\x)\in T_\nu(\x)  \mbox{ for $\mu$-almost all }\x\in \er^d
\end{aligned}
$$
Indeed, the first and the third equivalences follow from Proposition~\ref{P:Tnu},
the second one follows from the equality $\f\mu=\f g\nu$.
Let us prove the fourth one: The implication $\Rightarrow$ is clear as any $\nu$-null set is also $\mu$-null. To show the converse denote $A=\{\x\in \er^d\setsep \f(\x)g(\x)\notin T_\nu(\x)\}$. Then $0=\mu(A)=\int_A g\di\nu$. Since $g>0$ on $A$, we deduce $\nu(A)=0$.
Finally, the last equivalence follows from the fact that $g>0$ $\mu$-almost everywhere using the observation that $\f(\x)g(\x)\in T_\nu(\x) \Leftrightarrow\f(\x)\in T_\nu(\x)$ whenever $g(\x)>0$.

Now we are ready to prove the statement of the lemma.
By Proposition~\ref{P:Tnu} there is a sequence $(\f_n)$ in $L^1(\mu,E)$ such that
$$T_\mu(\x)=\overline{\{\f_n(\x)\setsep n\in\en\}}\mbox{ for $\mu$-almost all } \x\in \er^d.$$
By \eqref{eq:munu} we get $\f_n(\x)\in T_\nu(\x)$ for $\mu$-almost all $\x\in \er^d$ and each $n\in\en$.
We deduce that $T_\mu(\x)\subset T_\nu(\x)$ for $\mu$-almost all $\x\in \er^d$.  The converse is similar --  
there is a sequence $(\h_n)$ in $L^1(\nu,E)$ such that
$$T_\nu(\x)=\overline{\{\h_n(\x)\setsep n\in\en\}}\mbox{ for $\nu$-almost all } \x\in \er^d,$$
hence the equality holds, a fortiori, for $\mu$-almost all $\x\in \er^d$. Finally, $\h_n g\in L^1(\mu,E)$ and whenever $g(\x)>0$ we have
$$\h_n(\x)g(\x)\in T_\nu(\x)\Leftrightarrow \h_n(\x)\in T_\nu(x)\mbox{ and }\h_n(\x)g(\x)\in T_\mu(\x)\Leftrightarrow \h_n(\x)\in T_\mu(x).$$
Since $g>0$ $\mu$-almost everywhere, \eqref{eq:munu} implies $\h_n(\x)\in T_\mu(\x)$ for $\mu$-almost all $\x\in \er^d$ and each $n\in\en$.
We deduce that $T_\nu(\x)\subset T_\mu(\x)$ for $\mu$-almost all $\x\in \er^d$.
\end{proof}

\begin{prop}\label{P:projekcenaMsdiv}
The mapping $P$ is a linear projection of the space $\Ms(\er^d,E)$ onto $\Ms^{div}(E)$ and $\norm{P}\le C$.
\end{prop}

\begin{proof} Let us start by proving that $P$ is well defined. Suppose that $\mu_1,\mu_2\in\Ms(E)$ are nonnegative measures, $\f_1\in L^1(\mu_1,E)$, $\f_2\in L^1(\mu_2,E)$ and $\f_1\mu_1=\f_2\mu_2$. Let $\nu=\mu_1+\mu_2$ and let $g_1,g_2$ be the Radon-Nikod\'ym derivatives of $\mu_1,\mu_2$, respectively, with respect to $\nu$. Fix now $j\in\{1,2\}$. 
Then
\begin{equation}\label{eq:korektneDef}P_\nu(\f_j g_j)\nu=P_\nu(\f_j) g_j\nu=P_\nu(\f_j)\mu_j=P_{\mu_j}(\f_j)\mu_j.\end{equation}
Indeed, the first equality follows from the pointwise definition of $P_\nu$, the second one from the choice of $g_j$ and the last one from Lemma~\ref{L:rovnostae}. Since the left-hand side equals $P(\f_1 \mu_2) = P(\f_2 \mu_2)$, it follows that
$$P_{\mu_1}(\f_1)\mu_1=P_{\mu_2}(\f_2)\mu_2,$$
therefore $P$ is well-defined.

Let us continue by proving linearity of $P$. It is clear that $P(\alpha\bmu)=\alpha P(\bmu)$ for $\bmu\in \Ms(\er^d,E)$ and $\alpha\in\er$. It remains to show the additivity. Suppose $\bmu_1,\bmu_2\in \Ms(\er^d,E)$. 
Suppose $\bmu_1=\f_1\nu_1$ and $\bmu_2=\f_2\nu_2$.
Let $h_1,h_2$ be densities of $\nu_1,\nu_2$ with respect to $\nu_1+\nu_2$. Then
$$\begin{aligned}
P(\bmu_1+\bmu_2)&=P(\f_1\nu_1+\f_2\nu_2)=P((\f_1 h_1+\f_2 h_2)(\nu_1+\nu_2))
\\&=P_{\nu_1+\nu_2}(\f_1 h_1+\f_2 h_2)(\nu_1+\nu_2)
\\&=(P_{\nu_1+\nu_2}(\f_1 h_1)+P_{\nu_1+\nu_2}(\f_2 h_2))(\nu_1+\nu_2)
\\&=P_{\nu_1+\nu_2}(\f_1 h_1)(\nu_1+\nu_2)+P_{\nu_1+\nu_2}(\f_2 h_2)(\nu_1+\nu_2)
\\&=P_{\nu_1}(\f_1)\nu_1+P_{\nu_2}(\f_2)\nu_2=P(\bmu_1)+P(\bmu_2).\end{aligned}$$
Indeed, all the equalities are obvious except for the sixth one, where we use similar arguments as in \eqref{eq:korektneDef}.

It follows from Proposition~\ref{P:Tnu} that $P$ is a projection onto $\Ms^{div}(E)$. 
The estimate $\norm{P}\le C$ follows from Proposition~\ref{P:projekcenaL1nu}.
\end{proof}

\section{Proof of the main result and final remarks}

We are now ready to prove Theorem~\ref{T:main}. It is an immediate consequence of the following proposition (using Propositions~\ref{P:popisCKK} and~\ref{P:bidual}).

\begin{prop}
$L^1(\er^d,E)/Y(E)$ is complemented in $\Ma(\er^d,E)/X(E)^\perp$ by a projection with norm at most  
$d_{BM}(E,\ell^2_d)$.
\end{prop}

\begin{proof}
Let $\bR:\Ma(\er^d,E)\to\Ms(\er^d,E)$ be the linear operator defined in Section~\ref{s:vector}. Let $P:\Ms(\er^d,E)\to\Ms^{div}(E)$ be the projection defined in \eqref{eq:projekce}. Finally, let $J$ be the isometry from Lemma \ref{l:isomorfismusJ}. Define the projection $Q$ by setting
$$Q(\bmu+X(E)^\perp)= J^{-1}(PR\bmu+Y_\sigma(E)),\quad [\bmu]=\bmu+X(E)^\perp\in \Ma(\er^d,E)/X(E)^\perp.$$

First observe that $Q$ is well-defined. Indeed, if $\bmu\in X(E)^\perp$, then $R(\bmu)\in Y_\sigma(E)$ by Proposition \ref{p:projekceR}.
Since $Y_\sigma(E)\subset\Ms^{div}(E)$, $PR(\bmu)=R(\bmu)\in Y_\sigma(E)$. Hence $J^{-1}(PR\bmu+Y_\sigma(E))=0$.

Let $\f\in L^1(\er^d,E)$.  Then $R(\f\lambda)=\f\lambda\in \Ms^{div}(E)$, thus $PR(\f\lambda)=\f\lambda$ and so
$Q(\f\lambda+X(E)^\perp)=\f+Y(E)$. This shows that $Q$ is a projection onto $L^1(\er^d,E)/Y(E)$.
Finally, obviously $\norm{\bR}=1$, $\norm{P}\le d_{BM}(E,\ell^2_d)$ by Proposition~\ref{P:projekcenaMsdiv} and $\norm{J^{-1}}\le 1$ by Lemma~\ref{l:isomorfismusJ}. The estimate of the norm of $Q$ then follows immediately.
\end{proof}

The previous proposition completes the proof of the main result. The resulting projection is a composition of three mappings. It is natural to ask whether the result can be prove more easily, using some more `natural' approach. There are two such ways that one is tempted to try. However, none of them works. Let us explain it in more detail.

The first possibility is to try to use the projection $\bS$. The motivation for that is that $\bS$ is a projection of $\Ma(\er^d,E)$ onto $L^1(\er^d,E)$. If one succeeded to prove that $\bS(X(E)^\perp)\subset Y(E)$, one would get a projection of $\Ma(\er^d,E)/X(E)^\perp$ onto $L^1(\er^d,E)/Y(E)$ by factorizing the projection $\bS$. But this is not possible by Proposition~\ref{P:protipriklady} below.

The second possibility is to try to use the composition of $I-\bA_s$ with $\bR$.
Indeed, $\bR$ maps $\Ma(\er^d,E)$ onto $\Ms(\er^d,E)$ and $I-\bA_s$ is a projection of $\Ms(\er^d,E)$ onto $L^1(\er^d,E)$. If one succeeded to prove that $(I-\bA_s)(Y_\sigma(E))\subset Y(E)$, one would get a projection of $\Ma(\er^d,E)/X(E)^\perp$ onto $L^1(\er^d,E)/Y(E)$ by factorizing the projection $(I-\bA_s)\bR$. But this is not possible by the following proposition.

\begin{prop}\label{P:protipriklady} 
If $d\ge2$, then $\bS(X(E)^\perp)\not\subset Y(E)$ and $(I-\bA_s)(Y_\sigma(E))\not\subset Y(E)$.
\end{prop}

\begin{proof}
1. Assume that $\bS(X(E)^\perp)\subset Y(E)$. Since the validity of this inclusion does not depend on the concrete norm on $E$, suppose without loss of generality that $E=\ell^1_d$. In this case $\bS$ is an $L$-projection of $\Ma(\er^d,E)=(L^1(\er^d,E))^{**}$ onto $L^1(\er^d,E)$, hence $L^1(\er^d,E)$ is $L$-embedded. The assumption  $\bS(X(E)^\perp)\subset Y(E)$ then implies that the restriction of $\bS$ to $X(E)^\perp$ is an $L$-projection onto $Y(E)$. Since $Y(E)^{**}=Y(E)^{\perp\perp}=X(E)^\perp$, we conclude that $Y(E)$ is $L$-embedded.
Using \cite[Corollary 6.4 and remark (i) on p. 435]{pfitznerLemb} it follows that the unit ball of $Y(E)$ is closed in $L^1(\er^d,E)$ equipped with the topology of  local convergence in measure. But this contradicts \cite[Proposition 7]{GL}. This contradiction completes the proof of the first assertion.

2. Let us prove the second assertion. Since the concrete norm on $E$ plays no role in the assertion, we write simply $\er^d$ instead of $E$. Let $[a,b]\subset\er$ be a closed interval and let $\gamma:[a,b]\to \er^d$ be a one-to-one $C^1$-smooth curve. Let us define a measure $\bmu\in\Ms(\er^d,\er^d)$ by
$$\bmu(B) = \int_{\gamma^{-1}(B)} \gamma' =\left(\int_{\gamma^{-1}(B)} \gamma_1',\dots,\int_{\gamma^{-1}(B)} \gamma_d'\right),\quad B\in\B(\er^d).
$$
Then $\bmu\in\Ms(\er^d,\er^d)\cap\Mn(\er^d,\er^d)$ and, moreover, $\div\bmu=\varepsilon_{\gamma(a)}-\varepsilon_{\gamma(b)}$ in $\Dcp(\er^d)$ (by $\varepsilon_{\x}$ we denote the Dirac measure supported at $\x$. Indeed, if $\varphi\in\Dc(\er^d)$, we have
$$\begin{aligned}
\ip{\div\bmu}{\varphi}&=-\ip{\bmu}{\nabla\varphi}=-\sum_{i=1}^d \int \partial_i\varphi\di\mu_i
=-\sum_{i=1}^d \int_a^b \partial_i\varphi(\gamma(t))\gamma_i'(t)\di t
\\&=-\int_a^b (\varphi\circ\gamma)'(t)\di t=\varphi(\gamma(a))-\varphi(\gamma(b))=\ip{\varepsilon_{\gamma(a)}-\varepsilon_{\gamma(b)}}{\varphi}.
\end{aligned}
$$
Finally, by \cite[Proposition 3.4]{CKK} there is $\f\in L^1(\er^d,\er^d)$ (with compact support) such that $\div\f=\varepsilon_{\gamma(a)}-\varepsilon_{\gamma(b)}$ in $\Dcp(\er^d)$. It follows that $\f\lambda-\bmu\in Y_\sigma(\er^d)$ and $(I-\bA_s)(\f\lambda-\bmu)=\f\lambda\notin Y(\er^d)$ as $\gamma(a)\ne\gamma(b)$. This completes the proof.
\end{proof}

We do not know the answer to the following question. Note that a positive answer would yield a stronger version of Proposition~\ref{P:protipriklady}.

\begin{ques}
Suppose that $d\ge 2$. Is it true that $$\bS(X(E)^\perp)=(I-\bA_s)(Y_\sigma(E))=L^1(\er^d,E)\ ?$$
\end{ques}

\begin{remark}
The referee pointed out that our approach is related to the fact that the Leray projection on $L^1$ is not bounded. This is in a sense true, let us explain it a bit. If $p\in(1,\infty)$ and $d\in\en$, then the space $\{\f\in L^p(\er^d,\er^d)\setsep \div\f=0\}$ is complemented in $L^p(\er^d,\er^d)$ by a canonical projection whose kernel consists of gradients of functions from a suitable function space, see e.g. \cite[Section III.1]{galdi}. This canonical projection is called the Leray projection (sometimes the Leray-Helmholtz projection or the Helmholtz-Weyl projection). For $p=1$ such projection does not exist.

It is indeed related to the topic of our paper -- if $\{\f\in L^1(\er^d,\er^d)\setsep \div\f=0\}$ was complemented in $L^1(\er^d,\er^d)$,
then the space $\F(\er^d)$ which is isometric to the respective quotient by \cite{CKK} would be isomorphic to a complemented subspace of $L^1(\er^d,\er^d)$, hence obviously complemented in the bidual. 

However, it is not possible to proceed like that as by \cite{naoSch} the space $\F(\er^2)$ cannot be isomorphically embedded into any $L^1$ space.
It follows, in particular, that $\{\f\in L^1(\er^d,\er^d)\setsep \div\f=0\}$ is not complemented in $L^1(\er^d,\er^d)$.
\end{remark}

\section*{Acknowledgement}

We are grateful to the referee for a careful reading of the paper and for several useful comments.


\end{document}